\documentclass[reqno,11pt]{amsart}
\usepackage{amsmath,amssymb,amsthm,bm,amsfonts,mathrsfs,graphicx}
\usepackage{breqn}
\usepackage{cite}
\usepackage[mathcal]{euscript}
\usepackage[breaklinks=true]{hyperref}
\usepackage{cleveref,enumitem}
\usepackage{macrostxd}
\usepackage[english]{babel}
\usepackage{setspace}
\usepackage{braket}
\usepackage{latexsym,bm}
\usepackage{amscd}
\usepackage{epstopdf}
\usepackage[all]{xypic}
\usepackage[titletoc]{appendix}
\usepackage{xfrac,xcolor}
\usepackage[normalem]{ulem}

\newcommand\numberthis{\addtocounter{equation}{1}\tag{\theequation}}

    \makeatletter
    \def\step{
    \@ifnextchar[ \@step{\@noitemargtrue\@step[\@itemlabel]}}
    \def\@step[#1]{
        \item[#1]\textit{}\hspace*{\dimexpr-\labelwidth-\labelsep}
        }
    \makeatother
\setenumerate{fullwidth,itemindent=\parindent,
    listparindent=\parindent,itemsep=0ex,
    partopsep=0pt,parsep=0ex}

\begin{document}

 \onehalfspacing

\title[]
{Instability of the solitary waves for the 1d NLS with an attrictive delta potential in the degenerate case}

%
%
\author[]{Xingdong Tang}
 \address{\hskip-1.15em Xingdong Tang \hfill\newline School of Mathematics and Statistics, \hfill\newline Nanjing Univeristy of Information Science and Techenology, \hfill\newline Nanjing, 210044, China}
 \email{txd@nuist.edu.cn}

\author[]{Guixiang Xu}
 \address{\hskip-1.15em Guixiang Xu \hfill\newline School of Mathematical Sciences,
 	\hfill\newline Beijing Normal University,
 	\hfill\newline Laboratory of Mathematics and Complex Systems,
 	\hfill\newline Ministry of Education,
 	\hfill\newline Beijing, 100875, People's Republic of China.}
 \email{guixiang@bnu.edu.cn}



\begin{abstract}In this paper, we show the orbital instability of the solitary waves $Q_{\Omega}e^{i\Omega t}$ of the 1d NLS with an attractive delta potential ($\gamma>0$)
	\begin{equation*}
	 \i u_t+u_{xx}+\gamma\delta u+\abs{u}^{p-1}u=0, \; p>5, 
	\end{equation*} 
where $\Omega=\Omega(p,\gamma)>\frac{\gamma^2}{4}$ is the critical oscillation number and determined by
	\begin{equation*}
\frac{p-5}{p-1}
\int_{ \arctanh\sts{ \frac{\gamma}{2\sqrt{\Omega}} } }^{+\infty}
\sech^{\frac{4}{p-1}}\sts{y}\d y
=
{ \frac{\gamma}{ 2\sqrt{\Omega} } }\sts{ 1-\frac{\gamma^2}{4\Omega} }^{-\frac{p-3}{p-1}} \Longleftrightarrow  \mathbf{d}''(\Omega) =0.
	\end{equation*}
The classical convex method and Grillakis-Shatah-Strauss's stability approach in \cite{A2009Stab, GSS1987JFA1} don't work in this degenerate case, 
 and the argument here is motivated by those in \cite{CP2003CPAM, MM2001GAFA, M2012JFA, MTX2018, O2011JFA}. The main ingredients are to construct the unstable second order approximation near the solitary wave $Q_{\Omega}e^{i\Omega t}$ on the level set $\Mcal(Q_{\Omega})$  accoding to the degenerate structure of the Hamiltonian and to construct the refined Virial identity to show the orbital instability of the solitary waves $Q_{\Omega}e^{i\Omega t}$ in the energy space. Our result is the complement of the results in \cite{FOO2008AIHP} in the degenerate case.

\end{abstract}

\maketitle


\section{Introduction}
In this paper, we consider the  1d nonlinear Schr{\"o}dinger with a delta potential
\begin{equation}\label{dnls}
  \begin{cases}
    \i u_t+u_{xx}+\gamma\delta u+\mu\abs{u}^{p-1}u=0, & (t,x)\in\R_{+}\times\R, \\
    u(0,x)=u_0(x)\in H^1(\R), &
  \end{cases}
\end{equation}
where $u$ is a complex-valued function of $(t,x)$,
$\gamma\in \R\backslash\{0\} $, $\delta$ is the Dirac delta distribution at the origin, $\mu=\pm 1$ and $1< p< \infty$.
For $\gamma \not = 0$,  \eqref{dnls} appeares in various physical models with a point defect on the line, for instance, nonlinear optics \cite{GHW2004PhysD} and references therein.  For the case $\gamma<0$, it corresponds the repulsive delta potential,  while for the case  $\gamma>0$ it is attractive.

There are many results about \eqref{dnls}. Local well-posedness for \eqref{dnls} in the energy space $H^1\sts{\R}$ is well understood by Cazenave in \cite{C2003NLS},  Fukuizumi, Ohta and Ozawa in \cite{FOO2008AIHP} and Masaki, Murphy and Segata in \cite{MMS2018}. More precisely, we have
\begin{prop}[Local well-posedness in $H^1\sts{\R}$]
  For any $u_0\in H^1\sts{\R}$, there exists $T_{\max}$ with $0<T_{\max}\leq +\infty$ and a unqie solution 
  $u\in\Ccal\sts{ \left[0, T_{\max}\right), H^1\sts{\R} }$ for \eqref{dnls} satisfying  
  \begin{center}
    either $T_{\max}=+\infty$,
    or  $T_{\max}<+\infty$ and $ \displaystyle\lim\limits_{t\nearrow T_{\max}}\norm{\partial_x u\sts{t} }_{2}=+\infty$.
  \end{center}
  Moreover, the mass and the energy are conserved 
  under the flow generated by \eqref{dnls}, i.e., for any $t\in \left[0, T_{\max}\right)$, we have
  \begin{gather}
    \label{mass}
    \Mcal\sts{u\sts{t}}:=\frac{1}{2}\int_{\R}\abs{u(t,x)}^2\textup{d} x = \Mcal(u_0),
    \\
    \label{energy}
      \Ecal\sts{u\sts{t}}
       : =
      \int_{\R}\left[ \frac{1}{2}\abs{ u_{x}\sts{t,x} }^2-\frac{\gamma}{2} \delta(x)\abs{u\sts{t,x}}^2-\frac{\mu}{p+1} \abs{u(t,x)}^{p+1}   \right] \textup{d} x= \Ecal\sts{u_0}. 
  \end{gather}
\end{prop}

By the Gagliardo-Nirenberg inequality and the conservation laws, we have the global wellposedness of \eqref{dnls}  in the energy space $H^1(\R)$ for $1<p<5$. 

In addition, for the repulsive potential case $\gamma<0$, equation \eqref{dnls} is also studied from the point of view of scattering. Banica and Visciglia proved the global well-posedness and scattering result of  the energy solution of \eqref{dnls} for the defocusing mass-supercritical nonlinearity $\mu<0$, $p>5$  in \cite{BV2016JDE}.  Ikeda and Inui obtained the scattering result of the energy solution of \eqref{dnls} below the ground state threshold for the focusing mass-supercritical nonlinearity $\mu>0$, $p>5$ in \cite{II2017APDE}.  Recently, Masaki, Murphy and Segata  showed the decay and modified scatering result of the solution of \eqref{dnls} with small initial data for $p=3$ in a weighted space in \cite{MMS2017}. One can also refer the instability of the solitary waves of \eqref{dnls} for $p>1$ to  \cite{FJ2008DCDS, CFFKS2008PhysD}.

Such results are not expected for the attractive case $\gamma>0$ because of the existence  of the eigenvalue $-\frac{1}{4} \gamma^2$ of the Schr\"{o}dinger operator $-\partial^2_{x}-\gamma \delta$ (see \cite{DP2011IMRN, HMZ2007CMP, MMS2018} and the references therein).  In this paper,  we will focus on the attractive delta potential ($\gamma>0$) and the focusing  nonlinearity ($\mu=1$) and  consider the stability/instability of the nonlinear solitary wave solutions for \eqref{dnls} with the following form 
\begin{equation*}
  u(t,x)=\e^{\i\omega t}Q_{\omega}(x).
\end{equation*} 
It is easy to verify that $ Q_{\omega} $ satisfies
\begin{equation}
  \label{eqQ}
    -\partial_{x}^{2}{Q}_{\omega}(x)+\omega {Q}_{\omega}(x) -\gamma\delta(x){Q}_{\omega}(x)
    -\abs{Q_{\omega}(x)}^{p-1}{Q}_{\omega}(x)=0.
\end{equation}
For the case $\omega>\frac{\gamma^2}{4}$, there exists a unqie positive, radial symmetric solution to \eqref{eqQ} which can be explicitly described as following (see  \cite{FJ2008DCDS, FOO2008AIHP, GHW2004PhysD, CFFKS2008PhysD, MMS2018} )
\begin{equation}
  \label{Q}
    {Q}_{\omega}(x)=
    \left[ \frac{(p+1)\omega}{2}
      \sech^2( \frac{(p-1)\sqrt{\omega}}{2}\abs{x} +\arctanh( \frac{\gamma}{2\sqrt{\omega}}) ) \right]^{\frac{1}{p-1}}.
\end{equation}
The stability of ${Q}_{\omega}$  is a crucial problem during the study of the dynamics of the flow induced by  \eqref{dnls}. We firstly recall the definition of the orbital stability/instability in order to show the orbital stability/instability of the solitary waves in the energy space.
\begin{defi}
\label{def:stab}
  The solitary wave $\e^{\i\omega t} Q_{\omega}\sts{x}$ of \eqref{dnls} is said to be orbitally stable in $H^1\sts{\R}$ 
  if   for any $\alpha>0$, there exists $\beta=\beta(\alpha)>0$ such that for any solution
  $u\sts{t}$ to \eqref{dnls} with initial data $u_0\in\Ucal\sts{Q_{\omega}~,~\beta}$, we have
  \begin{equation*}
    u\sts{t}\in \Ucal\sts{Q_{\omega}~,~\alpha},\quad\text{for any }\; t\geq 0,
  \end{equation*}
  where
  \begin{equation}\label{tube}
    \Ucal\sts{Q_{\omega}~,~\alpha}
    =
    \Set{u\in H^1\sts{\R} | \inf_{\theta\in\R}\norm{u\sts{\cdot}- Q_{\omega}\sts{\cdot}\textup{e}^{\textup{i}\theta} }_{H^1}<\alpha }.
  \end{equation}
  Otherwise, the solitary wave $\e^{\i\omega t} Q_{\omega}\sts{x}$ is said to be orbitally unstable  in $H^1\sts{\R}$.
\end{defi}

For \eqref{dnls} with the cubic nonlinearity, Goodman, Holmes and Weinstein showed the orbital stability of the solitary waves $ \e^{\i\omega t}Q_{\omega}(x) $ with  $4\omega>\gamma^2 $ in the energy space $H^1\sts{\R}$ in \cite{GHW2004PhysD}. Later, by the Vakhitov-Kolokolov stability criteria in \cite{VK1973RQE} (see also  \cite{A2009Stab, GSS1987JFA1,SS1985CMP}), Fukuizumi, Ohta and Ozawa generalized the result  
to the case $p>1$ in \cite{FOO2008AIHP} (see also \cite{CFFKS2008PhysD}). More precisely, the following results hold:
\begin{enumerate}
  \item for any $p \in (1,5]$, the solitary waves $ \e^{\i\omega t}Q_{\omega}(x) $ with $\omega>\frac{\gamma^2}{4}$ are orbitally stable in $H^1\sts{\R}$;
  \item for any $p>5$, there exists $\Omega=\Omega(p,\gamma)>\frac{\gamma^2}{4}$, such that
      \begin{itemize}
        \item the solitary waves $ \e^{\i\omega t}Q_{\omega}(x) $ with $\omega\in(\frac{\gamma^2}{4},\Omega)$ are orbitally stable in $H^1\sts{\R}$ ;
        \item the solitary waves $ \e^{\i\omega t}Q_{\omega}(x) $ with $\omega>\Omega$ are orbitally unstable in $H^1\sts{\R}$,
      \end{itemize}
      where $\Omega(p,\gamma)$ is  determined by
      \begin{equation}
      \label{Omegaeq}
        \frac{p-5}{p-1}
        \int_{ \arctanh\sts{ \frac{\gamma}{2\sqrt{\Omega}} } }^{+\infty}
          \sech^{\frac{4}{p-1}}\sts{y}\d y
        =
        { \frac{\gamma}{ 2\sqrt{\Omega} } }\sts{ 1-\frac{\gamma^2}{4\Omega} }^{-\frac{p-3}{p-1}}
      \Longleftrightarrow  \mathbf{d}''(\Omega)=0.
      \end{equation}
\end{enumerate}

Above all,  only the critical oscillation case $\omega=\Omega(p,\gamma)$ for $p>5$ is left open, for which the Vakhitov-Kolokolov stability criteria breaks down because of the fact that $ \mathbf{d}''(\Omega)=0$, i.e., the degeneracy of the second order derivativce of the function $ \mathbf{d}(\omega)=S_{\omega}(Q_\omega)$ at $\omega=\Omega(p, \gamma)$. Fukuizumi, Ohta and Ozawa conjectured that  the solitary wave $ \e^{\i\omega t}Q_{\omega}(x) $ with $\omega=\Omega(p, \gamma)$  is orbitally unstable in \cite{FOO2008AIHP}.
The purpose of this paper is to prove this conjecture according to the observations in  \cite{CP2003CPAM, MM2001GAFA, M2012JFA, MTX2018, O2011JFA}. More precisely, we have the main result as following.
\begin{theo}
\label{mainthm}
  Let $\gamma>0$, $\mu=1$, $p>5$ and $\Omega>\frac{\gamma^2}{4}$ satisfy \eqref{Omegaeq}. The solitary waves $ \e^{\i\Omega t}Q_{\Omega}(x) $ of \eqref{dnls}
  is orbitally unstable in the energy space $H^1(\R)$.  More precisely, 
   there exist $\alpha_0>0$ and $\lambda_0>0$ such that if
  $$u_{0}\sts{x}= Q_{\Omega}\sts{x}+ \lambda~\phi_{\Omega} \sts{x} +\widetilde{\rho}\sts{\lambda}~Q_{\Omega}\sts{x},$$
 where $0<\lambda<\lambda_0$, $\phi_{\Omega} = \frac{\partial Q_{\omega}}{\partial \omega }|_{\omega=\Omega}$  and $\widetilde{\rho}\sts{\lambda}$ is chosen by the implicit function theorem such that
 $$\Mcal\sts{u_{0}}=\Mcal\sts{Q_{\Omega}}, $$
  then there exists $t_0=t_0(u_{0})$ such that the solution
  $u\sts{t}$ of \eqref{dnls} with initial data $u_{0}$ satisfies
  \begin{equation*}
  \inf_{\theta\in \R}\big\|u(t_0,\cdot)
  -
  Q_{\Omega}\sts{\cdot}\e^{\i\theta}\big\|_{H^1(\R)} \geq \alpha_0.
  \end{equation*}
\end{theo}

As stated above, the classical modulation analysis and the Virial type identity doesn't work once again in \cite{GSS1987JFA1,  GSS1990JFA, S1985NLKG, S1985NLS, W1985SJMA, W1986CPAM} because of the degenerate property of $ \mathbf{d}''\sts{\Omega}$, we now give more details about the refined modulation decomposition and the refined Virial identity.

Firstly, we use the following decomposition
\begin{equation}\label{decomp_v2}
u\sts{x} = \e^{-\i \theta }
\bsts{
	Q_{\Omega}
	+
	{\lambda}\phi_{\Omega}
	+
	\rho({ \lambda})Q_{\Omega}  + \eps
}\sts{ x },  \;\; \rho\sts{\lambda} = -\frac{\|\phi_{\Omega}\|^2_2}{2\|Q_{\Omega}\|^2_2}\cdot \lambda^2
\end{equation}
for the function $u$ in the $\eta_0$-tube $\Ucal\sts{Q_{\Omega}~,~\eta_0} $ of $Q_{\Omega}$
(see \eqref{tube} for the definition of the $\eta_0$-tube of $Q_{\Omega}$ ), the above refined decomposition is
related with the landscape of the action functional $\Scal_{\omega}$ near $Q_{\Omega}$.
\begin{enumerate}
	\item By the variational characterization of $Q_{\Omega}$, the action functional
	$\Scal_{\omega} $
	has the following properties
	\begin{align*}
	\Scal'_{\Omega}\sts{Q_{\Omega}} =0,{\;\; \Scal''_{\Omega}\sts{Q_{\Omega}} = \mathcal{L},}
	\end{align*}
	where the null space of the linearized operator
	${\mathcal L}$
	is characterized by
	$\text{Null}\sts{\mathcal L}=\text{span}\{\i Q_{\Omega}\}$.
	
	By the finite degenerate property of the function $ \mathbf{d}\sts{\Omega}=\Scal_{\Omega}\sts{Q_{\Omega}}$, we know that 
	\begin{align*}
     \mathbf{d}''(\Omega)=0, 
     \;\;\text{and}\;\; 
      \mathbf{d}^{\tprime}(\Omega)
	\neq
	0,
	\end{align*}
	where the first equality means that the mass conservation quantity
	$\Mcal\sts{u}
	=\Mcal\sts{u}$
	has the local equilibrium point $Q_{\Omega}$  along the curve $\{Q_{\Omega + \lambda }\}_{\lambda\in \R}$.
	\item Up to the phase rotation invariances, the first order approximation of $u$ to $Q_{\Omega}$ comes from the tangent vector $\varphi_{\Omega}$ of the curve $\{Q_{\Omega+\lambda }\}_{\lambda\in \R}$  at $Q_{\Omega}$,  and we have the following degenerate result
	\begin{align}\label{dir_xi}
	\Scal^{\dprime}_{\Omega}\sts{Q_{\Omega}}
	\sts{\phi_{\Omega}, ~\phi_{\Omega}} = -\left<Q_{\Omega}, \phi_{\Omega}\right>= 0.
	\end{align}
	
	\item Up to the  phase rotation invariances, the second order approximation of $u$ to $Q_{\Omega}$ is the direction $Q_{\Omega}  $, which
	is the steepest descent direction of the quantity $\Mcal\sts{u}$
	at $Q_{\Omega}$ along the curve $\{Q_{\Omega+\lambda }\}_{\lambda\in \R}$.  At the same time, we have the algebraic relations
	\begin{align*}\Scal''_{\Omega}\sts{Q_{\Omega}} \varphi_{\Omega} =  -Q_{\Omega}, \text{~~and ~~}
	\Scal'''_{\Omega}\sts{Q_{\Omega}}\sts{ \varphi_{\Omega} , \varphi_{\Omega} , \varphi_{\Omega} } + 3 \sts{\varphi_{\Omega} , \varphi_{\Omega} }= \mathbf{d}^{\tprime}(\Omega),
	\end{align*}

Now we take the following approximation \begin{align}\label{decomp_v1}
	Q_{\Omega}
	+
	{\lambda}\phi_{\Omega}
	+
	\rho({ \lambda})Q_{\Omega},
	\end{align}
	up to the  phase rotation invariances, where $\rho\sts{\lambda}$ can be ensured by restriction of the solution on the level set $\Mcal\sts{Q_{\Omega}}$ and indeed can be determined by the implicit function theorem (see Lemma \ref{lem:ift}).
	
	By the above approximation,  we can characterize the landscape of the function $\Scal_{\Omega}$ at $Q_{\Omega}$ along the perturbation $	{\lambda}\phi_{\Omega}
	+
	\rho({ \lambda})Q_{\Omega} $,
	\begin{equation*}
	\Scal_{\Omega}\sts{Q_{\Omega} + {\lambda}\phi_{\Omega}
		+
		\rho({ {\lambda}})Q_{\Omega}
	}
	=
	\Scal_{\Omega}\sts{Q_{\Omega}} +\frac{1}{6}\mathbf{d}^{\tprime}(\Omega)\cdot \lambda^3
	+\so{\abs{\lambda}^3},
	\end{equation*}
	\begin{equation*}
	\Scal_{\Omega}\sts{Q_{\Omega} + {\lambda}\phi_{\Omega}
		+
		\rho({ {\lambda}})Q_{\Omega}+\eps
	}
	=
	\Scal_{\Omega}\sts{Q_{\Omega}} +\frac{1}{6}\mathbf{d}^{\tprime}(\Omega)\cdot \lambda^3
	+\Scal^{\dprime}_{\Omega}\sts{Q_{\Omega}}\sts{\eps,~\eps}
	+\so{\abs{\lambda}^3+\norm{\eps}_{H^1}^2},
	\end{equation*}
	which means that if the small remainder term $\eps$ can be ignored,
	$\Scal_{\Omega}$ is a local monotone function with respect to $\lambda$
	under the special perturbation ${\lambda}\phi_{\Omega}
	+
	\rho({ {\lambda}})Q_{\Omega}$ near $Q_{\Omega}$ ,
	that is to say, the perturbation in the direction $\phi_{\Omega}$
	can play the dominant role under this special perturbation.
	This definite property of $\Scal_{\Omega}$ helps us to show the orbital instability of
	the solitary waves of \eqref{dnls} with the Virial argument in the degenerate case.
	
	\item The remainder $\varepsilon$ in \eqref{decomp_v2} is not only small,
	but also has some orthogonal structures, which makes the linearized operator
	$\Lcal=\Scal''_{\Omega}(Q_{\Omega})$ to possess almost coercivity to ensure the control of the remaider term $\eps$, see Lemma \ref{lem:coer}.
\end{enumerate}

Secondly, in order to show the orbital instability of the solitary waves
$Q_{\Omega}\sts{x}\e^{\i\Omega t}$ of \eqref{dnls},
we now turn to the effective monotonicity formula.
Since the quadratic term in $\lambda$ of
\begin{align*}
\action{Q_{\Omega} + {\lambda}~\phi_{\Omega}
		+
		\rho({ {\lambda}})~Q_{\Omega}
		+
		\eps
}{\varphi_{\Omega}},
\end{align*}
which corresponds to the term in \eqref{thett}, has the indefinite sign. By introducing the perturbation of $\varphi_{\omega, c}$ in the subspace $\text{Null}\sts{\mathcal{L}}$ to obtain the cancelation effect in the quadratic term  in $\lambda$ of \eqref{thett}, we can construct the refined Virial type quantity in the remainder term $\eps(t)$
\begin{equation}
\Iscr\sts{t} = \action{\i\eps\sts{t,x}}{ \varphi_{\Omega}\sts{x}
	-\lambda\sts{t}
	\frac{\braket{\varphi_{\Omega},\varphi_{\Omega}}
	}{
		\braket{ Q_{\Omega},Q_{\Omega}}
	}
	Q_{\Omega}\sts{x}} ,
\end{equation}
which has the monotone property in some sense (see \eqref{eIt}), to show the orbital instability of the solitary wave $Q_{\Omega}\sts{x}\e^{\i\Omega t}$ of \eqref{dnls}.

At last, the paper is organized as following. 
 In Section \ref{sect:pre},  we recall some properties of the linear Schr\"{o
}dinger operaotr with the dirac potential, the landscape of the action functional $\Scal_{\omega}$ at $Q_{\Omega}$ along the unstable dirction $\phi_{\Omega}$, and the refined modulation decomposition of the functions in the $\eta$-tube of $Q_{\Omega}$,  and the coercivity property of the linearized operator $\Lcal=\Scal_{\Omega}''\sts{Q_{\Omega}}$ on the subspace with the finite co-dimensions; In Section \ref{sect:eps-dyn}, we deduce the equation obeyed by the remainder term
$\eps\sts{t,x}$, and show the dynamical estimates of the parameters
$\lambda\sts{t}$ and $\theta\sts{t}$ by the geometric structures of
the remainder term. In Section \ref{sect:main},
we first construct the solutions of \eqref{dnls} near the solitary wave with
the refined geometric structures,
then show the orbital instability of the solitary wave of \eqref{dnls}
in the degenerate case by the dynamical behaviors of the remainder term and the parameters, and the refined Virial identity.

In Appendix \ref{app:d3rdd}, we calculate the third order derivative  $\mathbf{d}'''(\Omega) $ of $\mathbf{d}(\omega)=\Scal_{\omega}(Q_{\omega})$ at $\Omega$.

\subsection*{Acknowledgements.} 
The authors would like to thank Professor Thierry Cazenave and Professor Masahito Ohta for their valuable comments and suggestions, and the authors have been partially supported by the NSF grant of China (No. 11671046, and  No. 11831004). 

\section{Preliminaries}\label{sect:pre}
We make some preparations  in this section.
From now on, we fix $p>5$ and $\Omega=\Omega(p,\gamma)>\frac{\gamma^2}{4}$ is  determined by \eqref{Omegaeq}.
The Hilbert spaces $L^2\sts{\R,\C}$ and $H^1\sts{\R, \C}$ will be denoted by $L^2\sts{\R}$ and $H^1\sts{\R}$ respectively.
We denote 
\begin{equation*}
  \braket{u,v}=\Re\int u\sts{x}\bar{v}\sts{x}\d x,\quad\text{ for all } u,~v\in L^2\sts{\R}
\end{equation*}
be the inner product on the space $L^2\sts{\R}$.
For the simplification, we denote the following functions:
\begin{align*}
  f (z) =\abs{z}^{p-1}z, &\quad F(z)= \frac{1}{p+1}\abs{z}^{p+1}.
\end{align*}
A direct computation implies that for any $z_0,~z_1,~z_2,~z_3\in\C$, the following estimates hold: 
	\begin{align}
	\label{fexp1}
\abs{ f\sts{z_0+z_1} - f^{\prime}\sts{z_0}z_1 }\leq C\sts{\abs{z_0}}\sts{ \abs{z_1}^2 + \abs{z_1}^{p} },
\\
\label{fexp2}
\abs{ f\sts{z_0+z_1} - f^{\prime}\sts{z_0}z_1 -\frac{1}{2} f^{\dprime}\sts{z_0} {z_1z_1} }\leq C\sts{\abs{z_0}}\sts{ \abs{z_1}^3 + \abs{z_1}^{p} }, 
\end{align}
where $C\sts{\abs{z_0}}$ is a constant which only depends on $\abs{z_0}$, and
\begin{align}
\label{f1d}
f^{\prime}\sts{z_0}z_1
= & \;
\abs{z_0}^{p-1}z_1+\sts{p-1}\abs{z_0}^{p-3}\Re\sts{z_0\overline{z_1}}z_{0},
\\
\notag
f^{\dprime}\sts{z_0} {z_1z_2}
=
&\;
\sts{p-1}\abs{z_0}^{p-3}
\left[{ \Re\sts{ z_0 \overline{z_2} }z_1 + \Re\sts{ z_0 \overline{z_1} }z_2 +\Re\sts{ z_2 \overline{z_1} }z_0 }\right]
\\
\label{f2d}
&\;
+
\sts{p-1}\sts{p-3}\abs{z_0}^{p-5}\Re\sts{ z_0\overline{z_1} }\Re\sts{ z_0\overline{z_2} }z_0.
\end{align}
and 
\begin{align}
\label{F1std}
  F^{\prime}\sts{z_0}\sts{z_1} = &\;  \Re\left[{ f\sts{z_0}\bar{z}_1 }\right],
\\
   \label{F2ndd}
    F^{\dprime}\sts{z_0}\sts{z_1,z_2}
  =
  &\;
  \abs{z_0}^{p-3}
  \Big[ 
    p\Re \sts{z_0\bar{z}_1} \Re \sts{z_0\bar{z}_2} 
    + 
    \Im \sts{z_0\bar{z}_1} \Im \sts{z_0\bar{z}_2} 
  \Big],
\\
  F^{\tprime}\sts{z_0}\sts{z_1,z_2,z_3}
  =
  &\;
  \frac{p^2-1}{4}\abs{z_0}^{p-3}
  \Re\left[ z_0\sts{ \bar{z}_1\bar{z}_2 {z}_3 + \bar{z}_1{z}_2 \bar{z}_3 + {z}_1\bar{z}_2 \bar{z}_3 } \right] 
  \notag
  \\
  \label{F3rdd}
  &\; 
  +
  \frac{ \sts{p-1}\sts{p-3} }{ 4 }\abs{z_0}^{p-5} \Re\sts{ z_0\bar{z}_1  z_0\bar{z}_2 z_0\bar{z}_3 }.
\end{align}

%

\subsection{Linear Schr{\"o}dinger operator with a delta potential}\label{SubS:LSO}
We now recall some well-known properties for the linear Schr\"{o}dinger operator $ -\frac{\partial^2}{\partial x^2}-\gamma\delta$
with $\gamma\in\left[-\infty,+\infty \right)$, which were used in the physics literature. In fact, the following self-adjoint operator:
\begin{align*}
  &-{\Delta}_{\gamma} =
  -\frac{{\d}^2}{\d x^2}
  \text{~~with~~}
  \\
  &\text{D}(-\Delta_{\gamma})=
  \Set{\psi\in H^1\cap H^2(\R\setminus\Set{0}) | \frac{\d \psi}{\dx}(0+)-\frac{\d \psi}{\dx}(0-)=-\gamma \psi(0) }.
\end{align*}
gives the precise formulation of $ -\frac{\partial^2}{\partial x^2}-\gamma\delta$, see for instance  \cite{AGHH1988QM}. Moreover, the essential spectrum of $-{\Delta}_{\gamma} $ coincides with
$[0,+\infty)$. In addition, if $\gamma>0$,  $-{\Delta}_{\gamma}$ has exactly one
negative, simple eigenvalue, i.e. $-\frac{\gamma^2}{4}$ with the positive normalized eigenfunction
$ \sqrt{\frac{\gamma}{2}}\e^{-\frac{\gamma}{2}\abs{x}} $. Therefore, for any $\psi\in H^1(\R)$ and any
$\gamma>0$, we have
\begin{equation}
\label{eig1}
  -\frac{\gamma^2}{4}\int_{\R}\abs{\psi(x)}^2\dx \leq  \int_{\R}\abs{\psi'(x)}^2\dx-\gamma\int_{\R}\delta(x)\abs{\psi(x)}^2.
\end{equation}
As a consequence of the above inequality, we have
\begin{lemm}
  For any $\psi\in H^1(\R)$, the following inequality holds,
\begin{equation}
\label{sharp-embed}
  \abs{\psi(0)}^2 \leq  \norm{\psi'}_{2} \norm{\psi}_{2}.
\end{equation}
\begin{proof}
For any $\psi\in H^1(\R)$, since \eqref{eig1} holds for all $\gamma>0$, one can rewrite
\eqref{eig1} as
\begin{equation*}
  \int\delta(x)\abs{\psi(x)}^2 \leq   \frac{1}{\gamma}\int\abs{\psi'(x)}^2\dx+\frac{\gamma}{4}\int\abs{\psi(x)}^2\dx, \quad\text{~for all~} \gamma>0,
\end{equation*}
which implies \eqref{sharp-embed} by optimizing $\gamma$.
\end{proof}
\end{lemm}

\subsection{Basic properties of the action functional $\Scal_{\omega}$ and $d\sts{\omega}$ }\label{SubS:BP} For any $u\in H^1(\R)$, we define the action functional $\Scal_{\omega}$ as following:
\begin{equation}
\label{S}
\Scal_{\omega}(u)=\Ecal(u)+\omega\Mcal(u),
\end{equation}
where $\Mcal(u)$ and $\Ecal(u)$ are the  mass and energy  of $u$ defined by \eqref{mass}  and \eqref{energy} respectively. Since $p>5$, 
we have that
$\Scal_{\omega}$ is a $\Ccal^3$ functional on $H^1(\R)$ by \eqref{F1std}, \eqref{F2ndd} and \eqref{F3rdd}. In addition, we can obtain the following variational characterization of $Q_{\omega}$ by  the concentration-compactness argument in \cite{FOO2008AIHP,FJ2008DCDS,CFFKS2008PhysD}.
\begin{prop}
\label{vcsltn}
Let $\omega$ satisfy $4\omega>\gamma^2$. Then the function defined by \eqref{Q} 
is the unique positive, radial symmetric solution to \eqref{eqQ}. Moreover,
The set $\Set{ Q_{\omega}\e^{\i\theta} | \theta\in\R }$ coincides with all minimizers of the following minimization problem:
\begin{equation}
  \label{d}
  \inf\Set{
     \Scal_{\omega}\sts{ \psi } 
     | 
     \psi\in H^1\sts{\R}\setminus\Set{0},~~ 
     \Kcal_{\omega}\sts{ \psi } = 0  
  },
\end{equation}
where $\Scal_{\omega}$ is the action functional defined by \eqref{S}, and $\Kcal_{\omega}$ is the scaling derivative of $\Scal_{\omega}$ defined by 
\begin{align*}
  \mathcal{K}_{\omega}(\psi)
  =
  \left.\frac{\textup{d} }{\textup{d} \lambda}\Scal_{\omega}(\lambda\psi)\right|_{\lambda=1}
  =&
  \int \abs{\psi'(x)}^2+\omega\int \abs{\psi(x)}^2-{\gamma}\int\delta(x)\abs{\psi(x)}^2-\int\abs{\psi(x)}^{p+1}.
\end{align*}
\end{prop}
By the classical Weyl theorem in \cite{RS1980I} and \Cref{vcsltn}, one can give a precise description of the spectrum of the linearized operator
$\Scal^{\dprime}_{\omega}\sts{Q_{\omega}}:H^1\sts{\R}\times H^1\sts{\R}\mapsto\R$ which is self-adjoint operator and has the following form
\begin{align}
\notag
  \label{linS}
  \Scal^{\dprime}_{\omega}\sts{Q_{\omega}}\sts{f,g}
  =
  &
  \int
    \sts{ f_{1}^{\prime}{g_{1}}^{\prime} +\omega f_{1}{g_{1}}
    -\gamma\delta f_{1}{g_{1}} -p Q_{\omega}^{p-1}f_{1}{g_{1}}
    }
    \\
    &
    +
      \int
      \sts{
      f_{2}^{\prime}{g_{2}}^{\prime} +\omega f_{2}{g_{2}}
      -\gamma\delta f_{2}{g_{2}} - Q_{\omega}^{p-1}f_{2}{g_{2}}
      },
\end{align}
where $f_1$, $g_1$ are the real part of $f$,$g$ respectively, and $f_2$, $g_2$ are the imagination part of $f$,$g$ respectively. By the variational argument, 
Fukuizumi and Jeanjean obtained the following orthogonal decomposition about  $H^1(\R)$ according to the spectrum of the linearized operator $\Scal^{\dprime}_{\omega}\sts{Q_{\omega}}$ in \cite{FJ2008DCDS} (See also \cite{CFFKS2008PhysD}).
\begin{prop}
\label{prop:spectrum}
Let $\gamma>0$ and $\omega$ satisfy $4\omega>\gamma^2$. Then
the space $H^1\sts{\R}$ can be decomposed as the following direct sum 
\begin{equation}
\label{H1docom}
  H^1=N\oplus K\oplus P,
\end{equation}
according to the spectrum of the operator 
$\Scal^{\dprime}_{\omega}\sts{Q_{\omega}}$, where
\begin{enumerate}[label=\textup{(\roman*)}]
\item the subspace $N$, which is spanned by the eigenvector corresponding to the negative eigenvalue $-\mu^2$ of the operator $\Scal^{\dprime}_{\omega}\sts{Q_{\omega}}$, and is one dimensional,
  i.e. for any $f\in N$ with $f\neq 0$, we have
  \begin{equation*}
    \Scal^{\dprime}_{\omega}\sts{Q_{\omega}}\sts{f,f}=-\mu^2\braket{f,f};
  \end{equation*}
\item the subspace $K$ is the kernel (null) space for the operator $\Scal^{\dprime}_{\omega}\sts{Q_{\omega}}$, which is
  \begin{equation*}
    K=\textup{span}\Set{ \i Q_{\omega} };
  \end{equation*}
\item the subspace $P$ where the operator $\Scal^{\dprime}_{\omega}\sts{Q_{\omega}}$ has the coercivity, that is,  for any $f\in P$, we have
  \begin{equation*}
    \Scal^{\dprime}_{\omega}\sts{Q_{\omega}}\sts{f,f}
      \geq c \norm{f}_{H^1}^2,
  \end{equation*}
  where $c$ is a positive constant which does not depend on $f$.
\end{enumerate}
\end{prop}

Next, we turn to investigate some properties of 
$$\mathbf{d}\sts{\omega}=    \Scal_{\omega}\sts{Q_{\omega}}$$ which is related to the landscape of the action functional $\Scal_{\omega}$ around $Q_{\omega}$. By \Cref{vcsltn} and \eqref{eqQ}, we have
for all $\omega$ satisfying $4\omega>\gamma^2$,
\begin{align}
\label{d1omega}
  \mathbf{d}^{\prime}\sts{\omega}
  =& \Mcal\sts{Q_{\omega}}.
\end{align}

Moreover, let 
\begin{equation}
\label{phi}
\varphi_{\omega}\sts{x}=\frac{\partial Q_{\omega}}{\partial \omega}\sts{x},
\end{equation}
 we have
\begin{align}
\notag
\mathbf{d}^{\dprime}\sts{\omega}
=&
\frac{\partial^2}{\partial\omega^2}\Scal_{\omega}\sts{Q_{\omega}}
\\
\notag
=&
\braket{\Scal_{\omega}^{\prime}\sts{Q_{\omega}},\frac{\partial^2}{\partial\omega^2}Q_{\omega}}
+
\Scal^{\dprime}_{\omega}\sts{Q_{\omega}}
\sts{ \frac{\partial}{\partial\omega}Q_{\omega},\frac{\partial}{\partial\omega}Q_{\omega} }
+
2\braket{ Q_{\omega}, \frac{\partial}{\partial\omega}Q_{\omega}}
\\
=&
\Scal^{\dprime}_{\omega}\sts{Q_{\omega}}
\sts{ \varphi_{\omega} , \varphi_{\omega} }
+
2\braket{ Q_{\omega}, \varphi_{\omega} },
\label{d2omega}
\end{align}
where we used the fact that $\Scal_{\omega}^{\prime}\sts{Q_{\omega}}=0$.
Furthermore, we have
\begin{lemm}
\label{phi2Q}
Let $p>5$, $4\omega>\gamma^2$ and $\varphi_{\omega}$ defined by \eqref{phi}, the following result holds
  \begin{equation}
  \label{eq:phi2Q}
    \Scal_{\omega}^{\dprime}\sts{Q_{\omega}}\sts{\varphi_{\omega},\psi}=-\braket{Q_{\omega},\psi}, \quad \text{for any }
    \;  \psi\in H^1\sts{\R}.
  \end{equation}
\end{lemm}
\begin{proof}
  It is a well-known result and we can also refer to  Lemma $2.7$ in  \cite{FJ2008DCDS}. In fact,  It suffices to check that the following facts hold
  \begin{align}
    \label{phijump}
      \sts{\frac{\partial \varphi_{\omega}}{\partial x}}\sts{0+}
      -
      \sts{\frac{\partial \varphi_{\omega}}{\partial x}}\sts{0-}
      =
      -\gamma\varphi_{\omega}\sts{0},
  \end{align}
  and
  \begin{align}
    \label{eq-phi2Q}
      -\frac{\partial^2 \varphi_{\omega}}{\partial x^2}(x)+\omega \varphi_{\omega}(x) -p {Q_{\omega}(x)}^{p-1}\varphi_{\omega}(x)=-Q_{\omega}(x),
      \quad\text{ for all } x\neq 0.
    \end{align}

  On the one hand, since $Q_{\omega}$ can be explicitly expressed by \eqref{Q}, a direct computation implies that \eqref{phijump} holds.
  On the other hand, since $Q_{\omega}$ satisfies \eqref{eqQ}, 
we can obtain \eqref{eq-phi2Q}  by taking derivative with respect to $\omega$ in \eqref{eqQ}.
\end{proof}

As a consequence of  \eqref{d2omega} and \eqref{eq:phi2Q}, we know that
  \begin{equation}
  \label{deg:Q}
  \mathbf{d}^{\dprime}\sts{\Omega}=0 \iff
      \braket{ \varphi_{\Omega} , Q_{\Omega}}=0.
\end{equation}
It corresponds to the degenerate case, and  the Vakhitov-Kolokolov stability criterion in  \cite{VK1973RQE} (see also  \cite{A2009Stab, GSS1987JFA1,SS1985CMP}) breaks down in this case. In fact, one can still consider the stability (or instability) of the solitary waves
through the non-degenerate behavior of higher order derivative of $\mathbf{d}\sts{\omega}$  as those in \cite{MM2001GAFA}  (see also  \cite{CP2003CPAM, M2012JFA, MTX2018, O2011JFA}).
For this purpose, we first characterize the behavior of    $\mathbf{d}\sts{\omega}$ at the critical value $\Omega$.
\begin{lemm}
\label{lem:d3rdd}
  Let $p>5$ and $\Omega$ be defined by \eqref{Omegaeq}. Then we have the following results.
\begin{enumerate}
	\item For the case $\frac{\gamma^2}{4}< \omega < \Omega$, we have $   \mathbf{d}^{\dprime}\sts{\omega}  >0$;
	\item  For the case $ \omega = \Omega$, we have $ \mathbf{d}^{\dprime}\sts{\Omega}   =0, $ and $	\mathbf{d}^{\tprime}\sts{\Omega}<0$;
	\item  For the case $\omega>\Omega$, we have $  \mathbf{d}^{\dprime}\sts{\omega}  <0$.
\end{enumerate}
Furthermore, $\mathbf{d}^{\tprime}\sts{\Omega}$ can be explicitely expressed as following
  \begin{equation}
  \label{eq:d3rdd}
    \mathbf{d}^{\tprime}\sts{\Omega}
    =
    \Scal_{\Omega}^{\tprime}\sts{Q_{\Omega}}\sts{ \varphi_{\Omega},\varphi_{\Omega},\varphi_{\Omega} }
    + 
    3\braket{\varphi_{\Omega} ,\varphi_{\Omega} },
  \end{equation}
  where $\varphi_{\Omega}$ is defined by \eqref{phi}.
\end{lemm}
The proof of \Cref{lem:d3rdd} is postponed in  \Cref{app:d3rdd}.

\subsection{Geometric decomposition of $u$ and landscape of $\Scal_{\Omega}$ near $Q_{\Omega}$}\label{SubS:GD}
For the non-degenerate case, i.e. $ \mathbf{d}^{\dprime}\sts{\omega}< 0$ with $\omega>\Omega$,  the first order approximation of the solitary wave  in the unstable dirction  is enough to show the instability of the solitary waves, while for the degenerate case $\mathbf{d}^{\dprime}\sts{\Omega}= 0$ and $\mathbf{d}'''(\Omega) \not = 0$, we are going to consider the second order approximation of the solitary waves $Q_{\Omega}e^{i\Omega t}$,  up to the phase rotation invariance, on the level set $\Mcal(Q_{\Omega})$ to show its instability in the energy space. 
\begin{lemm}
\label{lem:ift}
There exist a constant  $0<\tilde{\lambda}_0 \ll 1$ and
a $\Ccal^2$ function $\tilde{\rho}:\sts{-\tilde{\lambda}_0,\tilde{\lambda}_0}\mapsto\R$ such that
for any $\lambda\in \sts{-\tilde{\lambda}_0,\tilde{\lambda}_0}$, we have 
\begin{equation*}
  \Mcal\sts{ Q_{\Omega}+\lambda \varphi_{\Omega}+\tilde{\rho}\sts{\lambda}Q_{\Omega} }
  =
  \Mcal\sts{Q_{\Omega}},
\end{equation*}
where the function $\tilde{\rho}(\lambda)$ can be expressed as following:
\begin{equation}
\label{func:trho}
  \tilde{\rho}\sts{\lambda}
  =
  -\frac{ \norm{ \varphi_{\Omega} }_{2}^{2} }{ 2\norm{Q_{\Omega}}_{2}^{2} }\lambda^2
  +\textup{o}\sts{\lambda^2},
  \quad
  \text{~for~any~} \lambda\in  \sts{-\tilde{\lambda}_0,\tilde{\lambda}_0}.
\end{equation}
\end{lemm}
\begin{proof}
 Essentially, the result is a consequence of the Implicit Function Theorem.
Let us define the function $G(\lambda, \rho)$ as following:
\begin{equation*}
  G\sts{\lambda,\rho}
  =
  \Mcal\sts{ Q_{\Omega}+\lambda \varphi_{\Omega}+\rho  Q_{\Omega} }
  -
  \Mcal\sts{Q_{\Omega}}.
\end{equation*}
By the simple calculations, one can obtain that $ G\sts{0,0}=0$, and  
\begin{equation}
\label{func:G00}
  \frac{\partial G}{\partial\lambda}\sts{0,0}=0,
  \quad
  \frac{\partial G}{\partial\rho}\sts{0,0}
  =\norm{Q_{\Omega}}_{2}^{2},
\end{equation}
and 
\begin{equation}
\label{func:G02}
\frac{\partial^2 G}{\partial\lambda^2}\sts{0,0}=\norm{ \varphi_{\Omega} }_{2}^{2},
\quad 
\frac{\partial^2 G}{\partial\lambda\partial\rho}\sts{0,0}=0,
\quad 
\frac{\partial^2 G}{\partial\rho^2}\sts{0,0}=\norm{ Q_{\Omega} }_{2}^{2},
\end{equation}
then by the Implicit Function Theorem, there exist a $\tilde{\lambda}_0$ with $0<\tilde{\lambda}_0 \ll 1$
and a $\Ccal^2$ function
$\tilde{\rho}:\sts{-\tilde{\lambda}_0,\tilde{\lambda}_0}\mapsto\R$ such that
\begin{equation}
\label{func:g}
  g\sts{\lambda}=G\sts{ \lambda,\tilde{\rho}\sts{\lambda} }=0,
  \quad\text{for all}\quad
  \lambda\in \sts{-\tilde{\lambda}_0,\tilde{\lambda}_0}.
\end{equation}
Therefore, it follows from \eqref{func:g} that
\begin{align}
\label{g1std}
  0=\frac{\d g}{\d \lambda}\sts{0}
  =\frac{\partial G}{\partial\lambda}\sts{0,0}
    +
    \frac{\partial G}{\partial\rho}\sts{0,0}\frac{\d\tilde{\rho}}{\d\lambda}\sts{0},
\end{align}
then by \eqref{func:G00},  we obtain
\begin{equation}
\label{rho1std00}
  \frac{\d\tilde{\rho}}{\d\lambda}\sts{0}=0.
\end{equation}
Again, by taking the second order derivative of the function $g$ with respect to $\lambda$  at $0$, we have
\begin{align}
\label{func:g02}
  0=\frac{\d^2 g}{\d \lambda^2}\sts{0}
  =
  \frac{\partial^2 G}{\partial\lambda^2}\sts{0,0}
  +
  2\frac{\partial^2 G}{\partial\lambda\partial\rho}\sts{0,0}\frac{\d \tilde{\rho}}{\d\lambda}\sts{0}
  +
  \frac{\partial G}{\partial\rho}\sts{0,0}\frac{\d^2 \tilde{\rho}}{\d\lambda^2}\sts{0}
\end{align}
By \eqref{func:G02} and \eqref{rho1std00}, we get
\begin{equation}
\label{func:rho2ndd00}
  \frac{\d^2 \tilde{\rho}}{\d\lambda^2}\sts{0}
  =
  -\frac{ \norm{ \varphi_{\Omega} }_{2}^{2} }{ \norm{ Q_{\Omega} }_{2}^{2} }.
\end{equation}
By the fundamental theorem of calculus, one can obtain the result. 
\end{proof}

From now on, we will take the function $\rho(\lambda)$ as the main part of  $\tilde{\rho}(\lambda)$, i.e.
\begin{equation}
  \label{func:rho}
    {\rho}\sts{\lambda}
    =
    -\frac{ \norm{ \varphi_{\Omega} }_{2}^{2} }{ 2\norm{Q_{\Omega}}_{2}^{2} }\lambda^2.
\end{equation}
Now, we can show the refined modulational decomposition of the functions around the solitary waves $Q_{\Omega}$.

\begin{lemm}
\label{lem:ift2}
There exists  $0<\tilde{\eta}_0\ll 1$ and a unique  $\Ccal^1$ map $(\theta, \lambda):  \Ucal\sts{ Q_{\Omega}, \tilde{\eta}_0 }\mapsto\R$ such that if $u\in  \Ucal\sts{ Q_{\Omega}, \tilde{\eta}_0 }$ and $\eps_{\theta, \lambda}(x)$ is defined by 
\begin{equation*}
  \eps_{\theta,\lambda}\sts{x}=u\sts{x}\e^{\i\theta}-\big(  Q_{\Omega}+\lambda\varphi_{\Omega}+\rho\sts{\lambda}Q_{\Omega} \big)(x)
\end{equation*}
where $\rho(\lambda)$ is define by \eqref{func:rho}, then we have the following orthogonal structure
\begin{equation*}
        \eps_{\theta,\lambda}\perp \i Q_{\Omega}
        \quad\text{and} \quad
        \eps_{\theta,\lambda}\perp \varphi_{\Omega}.
      \end{equation*}
Moreover, there exists a constant  $C$ which is independent of $\theta$, $\lambda$ and $u$, such that if  $u\in  \Ucal\sts{ Q_{\Omega}, \eta }$ with $\eta<\tilde{\eta}_0$, then we have
    \begin{equation*}
      \norm{\eps_{\theta,\lambda}}_{H^1}+\abs{\theta}+\abs{\lambda}\leq C\eta.
    \end{equation*}

\end{lemm}
\begin{proof}
It is also a consequence of the Implicit Function Theorem for the functional 
\begin{align*}
\textbf{F}(u; \theta, \lambda) = (F^1(u; \theta, \lambda) , F^2(u; \theta, \lambda) ),
\end{align*}
where 
\begin{align*}
F^1(u; \theta, \lambda)  = \; \Re \int    \eps_{\theta,\lambda} \; \overline{\i Q_{\Omega}} ,  \quad
F^2(u; \theta, \lambda)  = \;  \Re \int    \eps_{\theta,\lambda} \;  \varphi_{\Omega}.
\end{align*}
It suffices to verify the  non-degeneracy of the following Jacobian matrix:
\begin{equation*}
 \det \frac{\partial \textbf{F}}{\partial (\theta, \lambda) }\left(Q_{\Omega}, 0, 0\right) = \det
  \begin{pmatrix}
    \braket{\i Q_{\Omega},\i Q_{\Omega}} & \braket{\i Q_{\Omega},\varphi_{\Omega}}\\
    -\braket{\varphi_{\Omega},\i Q_{\Omega}} & -\braket{\varphi_{\Omega},\varphi_{\Omega}}
  \end{pmatrix}
  =
  -\norm{Q_{\Omega}}_{2}^{2}\norm{\varphi_{\Omega}}_{2}^{2} \not = 0.
\end{equation*}
We omit the details here. One can refer to \cite[Lemma 2.6]{MTX2018} for the analouge proof as the derivative NLS case.
\end{proof}
The next lemma shows that one can obtain the refined  estimate of  the remainder term $\eps_{\theta,\lambda}$
along the direction $Q_{\Omega}$ under the above refined modulational decomposition.
\begin{lemm}
\label{lem:epsQ}
There exist $0<\tilde{\eta}_1\ll 1$ and  $0<\tilde{\lambda}_1\ll 1$,
such that if $\abs{\lambda}\leq \tilde{\lambda}_1$ and any $\eps\in H^1(\R)$ with $\norm{\eps}_{H^1}\leq\tilde{\eta}_1$,
satisfy
\begin{equation*}
  \Mcal\sts{ Q_{\Omega}+\lambda\varphi_{\Omega}+\rho\sts{\lambda}Q_{\Omega} +\eps }=\Mcal\sts{ Q_{\Omega} },
\end{equation*}
where $\rho(\lambda)$ is define by \eqref{func:rho}, then we have
\begin{equation}
\label{epsQ}
  \abs{ \braket{ \eps,Q_{\Omega} } }
  \leq
  C\sts{ \norm{\eps}_{H^1}^2+\abs{\lambda}\norm{\eps}_{H^1} + \lambda^4 },
\end{equation}
where $C$ is a constant independent of $\lambda$ and $\eps$.
\end{lemm}
\begin{proof}
By $
  \Mcal\sts{ Q_{\Omega}+\lambda\varphi_{\Omega}+\rho\sts{\lambda}Q_{\Omega} +\eps }=\Mcal\sts{ Q_{\Omega} },
$ we have 
\begin{align}
  \notag
  0 =& \Mcal\sts{ Q_{\Omega}+\lambda\varphi_{\Omega}+\rho\sts{\lambda}Q_{\Omega} +\eps }-\Mcal\sts{ Q_{\Omega} }
  \\
  =& \braket{ Q_{\Omega} , \lambda\varphi_{\Omega}+\rho\sts{\lambda}Q_{\Omega} +\eps }
    +
    \frac{1}{2}\braket{ \lambda\varphi_{\Omega}+\rho\sts{\lambda}Q_{\Omega} +\eps, \lambda\varphi_{\Omega}+\rho\sts{\lambda}Q_{\Omega} +\eps }.
\end{align}
which together with \eqref{deg:Q} and \eqref{func:rho}  implies that
\begin{align}
\label{eq002}
\notag
  \braket{ \eps, Q_{\Omega} }
  =
  &
  -\lambda \braket{ \varphi_{\Omega} , \eps }
  -\frac{1}{2}\braket{ \rho\sts{\lambda}Q_{\Omega} +\eps , \rho\sts{\lambda}Q_{\Omega} +\eps }
  \\
  =&
  -\lambda \braket{ \varphi_{\Omega} , \eps }
  -\frac{\rho\sts{\lambda}^2}{2}\braket{ Q_{\Omega} , Q_{\Omega} }
  -\frac{1}{2}\braket{ \eps , \eps }
  -\rho\sts{\lambda}\braket{ \eps , Q_{\Omega} }.
\end{align}
By \eqref{func:rho} and the Cauchy-Schwarz inequality,  we can obtain the result.
%
\end{proof}
The following two lemmas show that the landscape  of the action functional $\Scal_{\Omega}(u)$  along the  unstable direction $\lambda \varphi_{\Omega} + \rho\sts{\lambda}Q_{\Omega} $ around the solitary wave $Q_{\Omega}$ is definite. Firstly we have
\begin{lemm}
\label{lem:expand1}
  There exist $0<\tilde{\lambda}_2\ll 1$, such that
if $ 0<\abs{\lambda}<\tilde{\lambda}_2 $,  we have
\begin{equation}
\label{eq:expand1}
  \Scal_{\Omega}\sts{ Q_{\Omega}+\lambda \varphi_{\Omega} + \rho\sts{\lambda}Q_{\Omega} }
  =
  \Scal_{\Omega}\sts{ Q_{\Omega} } + \frac{1}{6}\mathbf{d}^{\tprime}\sts{\Omega}\cdot \lambda^3 + \so{ \abs{\lambda}^3 },
\end{equation}
where $\rho(\lambda)$ is define by \eqref{func:rho}.
\end{lemm}
\begin{proof}
By the definition \eqref{func:rho} of $\rho(\lambda)$, $ \lambda \varphi_{\Omega} + \rho\sts{\lambda}Q_{\Omega} $
is sufficiently small in $H^1(\R)$. Therefore, by taking the Taylor series expression of $\Scal_{\Omega}$ at $Q_{\Omega}$, we have
\begin{align}
\label{eq003}
  &
  \Scal_{\Omega}\sts{ Q_{\Omega}+\lambda \varphi_{\Omega} + \rho\sts{\lambda}Q_{\Omega} }
  \\ \notag
  =
  &\; 
  \Scal_{\Omega}\sts{ Q_{\Omega} }
  +
  \frac{1}{2}\Scal_{\Omega}^{\dprime}\sts{ Q_{\Omega} }
  \sts{
    \lambda \varphi_{\Omega} + \rho\sts{\lambda}Q_{\Omega} ,
    \lambda \varphi_{\Omega} + \rho\sts{\lambda}Q_{\Omega}
  }
  \\
\notag
  &
  +
  \frac{1}{6}\Scal_{\Omega}^{\tprime}\sts{ Q_{\Omega} }
  \sts{
    \lambda \varphi_{\Omega} + \rho\sts{\lambda}Q_{\Omega} ,
    \lambda \varphi_{\Omega} + \rho\sts{\lambda}Q_{\Omega} ,
    \lambda \varphi_{\Omega} + \rho\sts{\lambda}Q_{\Omega}
  }
  \\
  \notag
  &
  +
  \so{ \norm{\lambda \varphi_{\Omega} + \rho\sts{\lambda}Q_{\Omega}}_{H^1}^3 }.
\end{align}

Firstly, it follows from  \eqref{F2ndd},  \eqref{eq:phi2Q},  \eqref{deg:Q} and \eqref{func:rho} that
\begin{align}
\label{eq004}
\frac{1}{2} \Scal_{\Omega}^{\dprime}\sts{ Q_{\Omega} }
  \sts{
    \lambda \varphi_{\Omega} + \rho\sts{\lambda}Q_{\Omega} ,
    \lambda \varphi_{\Omega} + \rho\sts{\lambda}Q_{\Omega}
  }
=
\frac{1}{2}
\lambda^3 \braket{ \varphi_{\Omega} , \varphi_{\Omega} }
+\so{ \abs{\lambda}^3 }.
\end{align}

Secondly, by  \eqref{func:rho} again, one can  get
\begin{align}
  \label{eq005}
  \Scal_{\Omega}^{\tprime}\sts{ Q_{\Omega} }
  \sts{
    \lambda \varphi_{\Omega} + \rho\sts{\lambda}Q_{\Omega} ,
    \lambda \varphi_{\Omega} + \rho\sts{\lambda}Q_{\Omega} ,
    \lambda \varphi_{\Omega} + \rho\sts{\lambda}Q_{\Omega}
  }
  =
  \lambda^3
  \Scal_{\Omega}^{\tprime}\sts{ Q_{\Omega} }
  \sts{
    \varphi_{\Omega},
    \varphi_{\Omega},
    \varphi_{\Omega}
  }
  +\so{\abs{\lambda}^3},
\end{align}
which together with \eqref{eq:d3rdd}, \eqref{eq003},  \eqref{eq004}   implies the result.
\end{proof}

Secondly, we have
\begin{lemm}
\label{lem:expand2}
There exist  $0<\tilde{\lambda}_3\ll 1$ and  $0<\tilde{\eta}_3\ll 1$ such that
if $\lambda$ satisfies $ 0<\abs{\lambda}<\tilde{\lambda}_3 $ and $\eps \in H^1(\R)$ with $ \norm{\eps}_{H^1}\leq \tilde{\eta}_3 $ satisfies
\begin{equation}
  \label{epsQ2}
    \abs{ \braket{ \eps,Q_{\Omega} } }
    \leq
    C\sts{ \norm{\eps}_{H^1}^2+\abs{\lambda}\norm{\eps}_{H^1} + \lambda^4 },
\end{equation}
where $C$ is a constant independent of $\lambda$ and $\eps$.
then we have
\begin{equation}
\label{eq:expand2}
  \Scal_{\Omega}\sts{ Q_{\Omega}+\lambda \varphi_{\Omega} + \rho\sts{\lambda}Q_{\Omega}+\eps }
  =
  \Scal_{\Omega}\sts{ Q_{\Omega} }+ \Scal_{\Omega}^{\dprime}\sts{ Q_{\Omega} }\sts{\eps,\eps} + \frac{1}{6}\mathbf{d}^{\tprime}\sts{\Omega}\cdot \lambda^3
  +
  \so{ \abs{\lambda}^3 +\norm{\eps}_{H^1}^2 },
\end{equation}
where $\rho(\lambda)$ is define by \eqref{func:rho}.
\end{lemm}
\begin{proof}
By the Taylor series  expansion of $\Scal_{\Omega}$ at $Q_{\Omega}$ and the fact that $\Scal_{\Omega}'\sts{ Q_{\Omega}} =0 $, we have
\begin{align}
\label{eq007}
&
  \Scal_{\Omega}\sts{ Q_{\Omega}+\lambda \varphi_{\Omega} + \rho\sts{\lambda}Q_{\Omega}+\eps }
  -
  \Scal_{\Omega}\sts{ Q_{\Omega} }
\\ 
\notag
=
&\;
\frac{1}{2}\Scal_{\Omega}^{\dprime}\sts{ Q_{\Omega} }
\sts{
  \lambda \varphi_{\Omega} + \rho\sts{\lambda}Q_{\Omega} +\eps,
  \lambda \varphi_{\Omega} + \rho\sts{\lambda}Q_{\Omega}+\eps
}
\\
\notag
&
+
\frac{1}{6}\Scal_{\Omega}^{\tprime}\sts{ Q_{\Omega} }
\sts{
  \lambda \varphi_{\Omega} + \rho\sts{\lambda}Q_{\Omega}+\eps ,
  \lambda \varphi_{\Omega} + \rho\sts{\lambda}Q_{\Omega}+\eps ,
  \lambda \varphi_{\Omega} + \rho\sts{\lambda}Q_{\Omega}+\eps
}
\\
\notag
&
+
\so{ \norm{\lambda \varphi_{\Omega} + \rho\sts{\lambda}Q_{\Omega}+\eps}_{H^1}^3 }.
\end{align}

Firstly, it follows from \eqref{eq:phi2Q} \eqref{func:rho}, \eqref{eq004}  and \eqref{epsQ2} that,
\begin{align}
\notag
 &  \frac{1}{2}\Scal_{\Omega}^{\dprime}\sts{ Q_{\Omega} }
  \sts{
  	\lambda \varphi_{\Omega} + \rho\sts{\lambda}Q_{\Omega} +\eps,
  	\lambda \varphi_{\Omega} + \rho\sts{\lambda}Q_{\Omega}+\eps
  } \\
=& \notag
  \frac{1}{2}
  \lambda^3 \braket{ \varphi_{\Omega} , \varphi_{\Omega} }
  +
  \frac{1}{2}
  \Scal_{\Omega}^{\dprime}\sts{ Q_{\Omega} }
  \sts{
    \eps,
    \eps
  }
  +
  \Scal_{\Omega}^{\dprime}\sts{ Q_{\Omega} }
  \sts{
    \lambda \varphi_{\Omega} + \rho\sts{\lambda}Q_{\Omega},
    \eps
  }
  +\so{ \abs{\lambda}^3 }
\\
\notag
=&
\frac{1}{2}
\lambda^3 \braket{ \varphi_{\Omega} , \varphi_{\Omega} }
+
\frac{1}{2}
\Scal_{\Omega}^{\dprime}\sts{ Q_{\Omega} }
\sts{
  \eps,
  \eps
}
-
\lambda\braket{ Q_{\Omega} , \eps }
+
\bo{ \lambda^2\norm{\eps}_{H^1} }
+\so{ \abs{\lambda}^3 }
\\
\notag
=&
\frac{1}{2}
\lambda^3 \braket{ \varphi_{\Omega} , \varphi_{\Omega} }
+
\frac{1}{2}
\Scal_{\Omega}^{\dprime}\sts{ Q_{\Omega} }
\sts{
  \eps,
  \eps
}
+\bo{ \abs{\lambda}\norm{\eps}_{H^1}^2+{\lambda}^2\norm{\eps}_{H^1} + \abs{\lambda}^5 }
+\so{ \abs{\lambda}^3 }
\\
\label{eq009}
=&
\frac{1}{2}
\lambda^3 \braket{ \varphi_{\Omega} , \varphi_{\Omega} }
+
\frac{1}{2}
\Scal_{\Omega}^{\dprime}\sts{ Q_{\Omega} }
\sts{
  \eps,
  \eps
}
+
\so{ \abs{\lambda}^3 +\norm{\eps}_{H^1}^2 }
\end{align}
where we used the Cauchy-Schwarz inequality in the last identity.

Secondly, by \eqref{eq005} and \eqref{func:rho}, we get
\begin{align}
  \notag
& \frac{1}{6}\Scal_{\Omega}^{\tprime}\sts{ Q_{\Omega} }
\sts{
	\lambda \varphi_{\Omega} + \rho\sts{\lambda}Q_{\Omega}+\eps ,
	\lambda \varphi_{\Omega} + \rho\sts{\lambda}Q_{\Omega}+\eps ,
	\lambda \varphi_{\Omega} + \rho\sts{\lambda}Q_{\Omega}+\eps
}\\
  =
  & \notag
  \frac{\lambda^3}{6}
  \Scal_{\Omega}^{\tprime}\sts{ Q_{\Omega} }
  \sts{
    \varphi_{\Omega},
    \varphi_{\Omega},
    \varphi_{\Omega}
  }
  +\so{\abs{\lambda}^3}
  +\bo{ \lambda^2\norm{\eps}_{H^1}+\norm{\eps}_{H^1}^3 }
  \\
  =
  &
  \frac{\lambda^3}{6}
  \Scal_{\Omega}^{\tprime}\sts{ Q_{\Omega} }
  \sts{
    \varphi_{\Omega},
    \varphi_{\Omega},
    \varphi_{\Omega}
  }
  +\so{\abs{\lambda}^{3}+\norm{\eps}_{H^1}^2}
\label{eq010}
\end{align}

Lastly, by  \eqref{eq007}, \eqref{eq009}, \eqref{eq010} and \eqref{eq:d3rdd},  we can obtain the result.
\end{proof}

\subsection{Properties of the linearized operator $ \Scal_{\Omega}^{\dprime}\sts{ Q_{\Omega} } $}\label{SubS:LO}
As shown in Lemma \ref{lem:expand2}, we now turn to estimate the quadratic term $ \Scal_{\Omega}^{\dprime}\sts{ Q_{\Omega} }\sts{\eps,\eps} $, which in fact  has some coercivity property under the condition that the remainder term $\eps$ has some geometric orhtogonal structures. It is the task in this subsection and related to the spectral properties of the linearized operators $\Scal_{\Omega}^{\dprime}\sts{ Q_{\Omega} } $.

To do so, we firstly introduce the following result.
\begin{lemm}
  \label{lem:chiRphi}
    Let $\chi$ be the $L^2$-normalized function of $N$  defined by \eqref{H1docom}, and $\varphi_{\Omega}$ be defined by \eqref{phi}.
    Then we have
  \begin{equation}
  \label{chiRphi}
    \braket{\chi,\varphi_{\Omega}}\neq 0.
  \end{equation}
  \end{lemm}
  \begin{proof}
    We argue by contradiction, and assume that
  \begin{equation}
  \label{eq-chiRphi-001}
    \braket{\chi,\varphi_{\Omega}}= 0.
  \end{equation}
  Since $\varphi_{\Omega}$ is real, it is easy to see that
    \begin{equation}
      \label{eq-chiRphi-002}
      \braket{ \varphi_{\Omega},\i Q_{\Omega} }=0.
    \end{equation}
  On the one hand, by \eqref{eq-chiRphi-001}, \eqref{eq-chiRphi-002} and \Cref{prop:spectrum}, we have
  \begin{equation}\label{Est:Quad Pos}
    \Scal_{\Omega}^{\dprime}\sts{Q_{\Omega}}\sts{ \varphi_{\Omega} , \varphi_{\Omega} } \geq c \norm{\phi_{\Omega}}_{H^1}^2>0.
  \end{equation}
  On the other hand,  by \eqref{eq:phi2Q} and \eqref{deg:Q}, we get
  \begin{equation*}
    \Scal_{\Omega}^{\dprime}\sts{Q_{\Omega}}\sts{ \varphi_{\Omega} , \varphi_{\Omega} } = -\braket{Q_{\Omega}, \phi_{\Omega}}=0,
  \end{equation*}
  which is in contradiction with \eqref{Est:Quad Pos}. Therefore, \eqref{chiRphi} holds, and this completes the proof.
  \end{proof}
After the above lemma, one can now show  the following coercive property of $\Scal_{\Omega}^{\dprime}\sts{Q_{\Omega}}$ by the standard arguments in \cite{GSS1987JFA1}  \cite{W1985SJMA}, 
which is a consequence of \Cref{prop:spectrum}.
\begin{lemm}
\label{lem:coer}
  Let $\eps\in H^1\sts{\R}\setminus\Set{0}$. If
  \begin{equation}
  \label{orth1}
    \braket{ \eps,\i Q_{\Omega } }=0, \quad \braket{ \eps, \varphi_{\Omega} }=0 \quad\text{ and }\quad \braket{ \eps, Q_{\Omega} }=0,
  \end{equation}
  then there exists a positive constant $\kappa_1$ independent of $\eps$, such that the following result holds,
  \begin{equation}
  \label{eq:coerh1}
    \Scal_{\Omega}^{\dprime}\sts{Q_{\Omega}}\sts{ \eps,\eps }\geq \kappa_1 \norm{ \eps }_{H^1}^2.
  \end{equation}
\end{lemm}
\begin{proof}
It suffices  to show that there exists a positive constant $\kappa_2$ independent of $\eps$, such that the following estimate holds,
\begin{equation}
\label{eq:orthl2}
  \Scal_{\Omega}^{\dprime}\sts{Q_{\Omega}}\sts{ \eps,\eps }\geq \kappa_2 \norm{ \eps }_{2}^2.
\end{equation}
In fact, assume that \eqref{eq:orthl2} holds, it follows from \eqref{linS} and $ \norm{Q_{\Omega}}_{\infty}<+\infty $ that
\begin{align}
\label{eq013}
  \norm{ \eps^{\prime} }_{2}^2+\Omega \norm{\eps}_{2}^{2}-\gamma\abs{ \eps\sts{0} }^2
  \leq
  \Scal_{\Omega}^{\dprime}\sts{Q_{\Omega}}\sts{{\eps},{\eps} }
  + C\norm{\eps}_{2}^{2},
\end{align}
where $  C $ is a positive constant which only depends on $\norm{Q_{\Omega}}_{\infty}$.
Using \eqref{sharp-embed}, the Cauchy-Schwarz inequality and the fact $4\Omega>\gamma^2$, we have
\begin{align}
\notag
\textup{LHS of }\eqref{eq013}
=
& \norm{ \eps^{\prime} }_{2}^2+\Omega \norm{\eps}_{2}^{2}
  -\gamma\abs{ \eps\sts{0} }^2
\\
\notag
\geq
& \norm{ \eps^{\prime} }_{2}^2+\Omega \norm{\eps}_{2}^{2}-\frac{1}{2}\norm{ \eps^{\prime} }_{2}^2-\frac{\gamma^2}{2}\norm{ \eps }_{2}^2
\\
\label{eq014}
\geq
& \frac{1}{2}\norm{ \eps^{\prime} }_{2}^2 -\frac{\gamma^2}{2}\norm{ \eps }_{2}^2.
\end{align}
Therefore, inserting \eqref{eq:orthl2} and \eqref{eq014} into \eqref{eq013}, one immediately get
\begin{align*}
  \frac{1}{2}\norm{ \eps^{\prime} }_{2}^2
  \leq &
  \frac{\gamma^2}{2}\norm{ \eps }_{2}^2 +
  \Scal_{\Omega}^{\dprime}\sts{Q_{\Omega}}\sts{{\eps},{\eps} }
  + C \norm{\eps}_{2}^{2}
  \\
  \leq &
  \sts{1+ \frac{ \gamma^2+2C }{ 2\kappa_2} }
  \Scal_{\Omega}^{\dprime}\sts{Q_{\Omega}}\sts{{\eps},{\eps} }
\end{align*}
By taking $\kappa_1 = \frac{ \kappa_2 }{ 2\kappa_2  +\gamma^2+2C + 1}$, we can obtain \eqref{eq:coerh1}.

Now let $\chi$ be the  $L^2$-normalized function in $N$.
Firstly, for any nonzero $\eps$ satisfying \eqref{orth1},  one can take the following decomposition by \eqref{chiRphi}:
  \begin{equation*}
    \eps = \mathtt{p}_{\eps} + a_{\eps}\varphi_{\Omega}, \quad  a_{\eps} = -\frac{ \braket{\eps,\chi} }{ \braket{ \varphi_{\Omega},\chi } }.
  \end{equation*}

On the one hand, a direct calculation implies that
\begin{equation*}
  \braket{ \mathtt{p}_{\eps},\chi }=0,\quad\text{and}\quad \braket{ \mathtt{p}_{\eps},\i Q_{\Omega}}=0,
\end{equation*}
which means that $ \mathtt{p}_{\eps} \in P$, where $P$ is defined by \Cref{prop:spectrum}, therefore we have
\begin{align}
  \notag
  \Scal_{\Omega}^{\dprime}\sts{Q_{\Omega}}\sts{ \mathtt{p}_{\eps},\mathtt{p}_{\eps} }
  \geq& c \norm{ \mathtt{p}_{\eps} }_{2}^{2}
\\
\notag
  =&
  c \braket{ \eps-a_{\eps}\varphi_{\Omega} , \eps-a_{\eps}\varphi_{\Omega} }
  \\
\notag
  =&
  c \norm{\eps}_{2}^{2}+c \sts{a_{\eps}}^2\norm{\varphi_{\Omega}}_{2}^{2}
  \\
\label{eq011}
  \geq& c \norm{\eps}_{2}^{2}.
\end{align}

On the other hand, by \eqref{eq:phi2Q}, \eqref{deg:Q} and \eqref{orth1}, we have
\begin{align}
  \Scal_{\Omega}^{\dprime}\sts{Q_{\Omega}}\sts{{\eps},{\eps} }
  =
  \Scal_{\Omega}^{\dprime}\sts{Q_{\Omega}}\sts{ \mathtt{p}_{\eps}+a_{\eps}\varphi_{\Omega},\mathtt{p}_{\eps}+a_{\eps}\varphi_{\Omega} }
  =
  \Scal_{\Omega}^{\dprime}\sts{Q_{\Omega}}\sts{ \mathtt{p}_{\eps} ,\mathtt{p}_{\eps} }.
  \label{eq012}
\end{align}
Combining \eqref{eq011} and \eqref{eq012}, we can obtain
\begin{equation*}
  \Scal_{\Omega}^{\dprime}\sts{Q_{\Omega}}\sts{{\eps},{\eps} }\geq c \norm{\eps}_{2}^{2}.
\end{equation*}
This completes the proof of \eqref{eq:orthl2} with $\kappa_2=c$ and the proof of \Cref{lem:coer}.
\end{proof}

As a consequence of \Cref{lem:coer}, we have
\begin{coro}
    \label{coro:coer}
      Let $\eps\in H^1\sts{\R}\setminus\Set{0}$ and satisfy 
      \begin{equation}
      \label{orth2}
        \braket{ \eps,\i Q_{\Omega } }=0 \quad\text{ and }\quad \braket{ \eps, \varphi_{\Omega} }=0,
      \end{equation}
      then there exists a positive constant $\kappa$ independent of $\eps$, such that the following estimate holds,
      \begin{equation}
      \label{eq:coerh2}
        \Scal_{\Omega}^{\dprime}\sts{Q_{\Omega}}\sts{ \eps,\eps }\geq {\kappa} \norm{ \eps }_{H^1}^2-\frac{1}{\kappa}\braket{ \eps, Q_{\Omega} }^2.
      \end{equation}
\end{coro}
\begin{proof}
  The proof  is standard, we can refer the analouge proof as  the nonlinear Schr\"{o}dinger equation in \cite[page 186]{MR2005AM}.  
\end{proof}

Combining Lemma \ref{lem:epsQ} and Corollary \ref{coro:coer}, we have
\begin{lemm}
\label{lem:restcoer}There exist $0<\tilde{\eta}_4\ll 1$ and $0<\tilde{\lambda}_4\ll 1$ such that if $\eps\in H^1\sts{\R}$
with $\norm{\eps}_{H^1}\leq\tilde{\eta}_4 $,
and $\lambda$ with
$\abs{ {\lambda}}\leq \tilde{\lambda}_4 $ satisfy
\begin{equation*}
  \action{\eps}{\textup{i}{Q_{\Omega}}}=0 \quad
  \text{~~and~~}
  \action{\eps}{\phi_{\Omega}}=0,
\end{equation*}and
\begin{equation*}
\Mcal\sts{Q_{\Omega} + {\lambda}\phi_{\Omega}
        +
        \rho({ {\lambda}})Q_{\Omega}+\eps
    }=\Mcal\sts{Q_{\Omega}},
\end{equation*}
where $\rho(\lambda)$ is determined  by \eqref{func:rho}, then we have
\begin{equation*}
  {\Scal}_{\Omega}^{\dprime}\sts{ Q_{\Omega} }\sts{\eps,\eps}\geq \frac{\kappa}{2}\norm{\eps}_{H^1}^2
  +
  \so{ {\lambda}^4}.
\end{equation*}
\end{lemm}
\begin{proof}
First, by \Cref{lem:epsQ}, we have
\begin{align}
\notag
\action{\eps}{Q_{\Omega}}^2
=
&
C^2\bsts{
    {
        \abs{ {\lambda}}\norm{\eps}_{H^1}
        +
        \norm{\eps}_{H^1}^2
        +
        {\lambda}^4
    }
  }^2
\\
\leq
&
3C^2\sts{
    \lambda^2\norm{\eps}_{H^1}^2
    +
    \norm{\eps}_{H^1}^4
    +
    {\lambda}^8
}.
\label{tp11}
\end{align}
It implies by taking $\abs{\lambda}$ and $\norm{\eps}_{H^1}$ sufficiently small  that
\begin{equation}
\label{tp1}
  \action{\eps}{Q_{\Omega}}^2=\so{\norm{\eps}_{H^1}^2}
  +
  \so{\lambda^4},
\end{equation}
which together with \eqref{eq:coerh2} implies that
  \begin{align*}
    {\Scal^{\dprime}_{\Omega}}\sts{Q_{\Omega}}\sts{\eps,\eps}
    \geq
    &
     \kappa\norm{\eps}_{H^1}^2
      -
      \frac{1}{\kappa}\action{\eps}{Q_{\Omega}}^2
    \\
   =
    &
    \kappa\norm{\eps}_{H^1}^2
    +
    \so{\norm{\eps}_{H^1}^2}
    +
    \so{\lambda^4}
    \\
    \geq
    &
    \frac{\kappa}{2}\norm{\eps}_{H^1}^2
    +
    \so{ {\lambda}^4}.
  \end{align*}
This completes the proof.
\end{proof}

\section{The $\eps$-variable equation and the dynamics of the parameters}\label{sect:eps-dyn}
In this section,  we derive the equation obeyed by the remainder term
  \begin{equation}    \label{epst}
\eps\sts{t,x}
=
u\sts{t,x}\e^{\i\theta(t)}-
\left({ Q_{\Omega}\sts{x}+\lambda\sts{t}\varphi_{\Omega}\sts{x}+\rho\sts{\lambda\sts{t}}Q_{\Omega}\sts{x} }\right),
\end{equation}
where $u$ is the solution of \eqref{dnls} in $H^1(\R)$, $\varphi_{\Omega}$ and $\rho(\lambda)$ are determined by \eqref{phi} and \eqref{func:rho} respectively,
$\lambda$ and $\theta$ are two $\Ccal^1$ functions with respect to $t$.

Firstly, we have 
\begin{lemm}
\label{lem:epst}
  Let $u(t) \in C\left([0, T), H^1(\R)\right)$ be the  solution to \eqref{dnls} for some $T>0$, and $\eps(t,x)$ be defined by \eqref{epst}, then we have
 
  \begin{align}
  \notag
    \i\eps_{t}
    =&
    -\i\lambda_{t}\sts{ \varphi_{\Omega}+\frac{\d\rho}{\d\lambda}\sts{\lambda} Q_{\Omega} }
    -\sts{ \theta_{t} + \Omega }\sts{ Q_{\Omega}+\lambda\varphi_{\Omega}+\rho\sts{\lambda}Q+\eps }
    \\
    \notag
    &
    + \Lcal\sts{ \lambda\varphi_{\Omega}+\rho\sts{\lambda}Q_{\Omega}+\eps }
    -\frac{1}{2}f^{\dprime}\sts{Q_{\Omega}}\sts{ \lambda\varphi_{\Omega}
    +\rho\sts{\lambda}Q_{\Omega}+\eps , \lambda\varphi_{\Omega}+\rho\sts{\lambda}Q_{\Omega}+\eps }
    \\
    \label{eq:epst}
    &
    -\Rcal\sts{ Q_{\Omega}, \lambda\varphi_{\Omega}+\rho\sts{\lambda}Q_{\Omega}+\eps},
  \end{align}
where $f^{\prime}\sts{Q_{\Omega}}$, 
$f^{\dprime}\sts{Q_{\Omega}}$ is defined by \eqref{f1d}, \eqref{f2d},
the linearized operator $\Lcal$ and the higher order remainder term $\Rcal$ are defined by
\begin{equation}
\label{op:L}
  \Lcal{ g } = \Scal^{\dprime}_{\Omega}\sts{Q_{\Omega}} g
  =
  -g_{xx}+\Omega g-\gamma\delta g -f^{\prime}\sts{Q_{\Omega}}g,
\end{equation}
and
\begin{equation}
\label{eq:hot}
  \Rcal\sts{ Q_{\Omega},g }
  =
  f\sts{Q_{\Omega}+g}-f\sts{Q_{\Omega}}-f^{\prime}\sts{Q_{\Omega}}g
  -\frac{1}{2}f^{\dprime}\sts{Q_{\Omega}}\sts{g, g} .
\end{equation}
\end{lemm}
\begin{proof}
First, let
$
  v\sts{t,x}
  =
  u\sts{t,x}\e^{ \i\theta\sts{t} }
$,
then we have
\begin{equation}
\label{eq3001}
  u_{t}=\sts{ v_{t}-\i\theta_{t} v }\e^{ -\i\theta }
  \quad\text{and}\quad
  u_{xx}=v_{xx}\e^{ -\i\theta },
\end{equation}
which together with \eqref{dnls} implies that
\begin{equation}
\label{eq3002}
  \i v_{t} + \theta_{t}v+v_{xx}+\gamma\delta v+f\sts{v}=0.
\end{equation}

Next, let
\begin{equation}
\label{eq3003}
  v\sts{t,x}=Q_{\Omega}\sts{x}+g\sts{t,x}.
\end{equation}
By \eqref{eq3002}, we have
\begin{align*}
  0
  =
  &
  \i g_{t}
  +
  \theta_{t}\sts{ Q_{\Omega} + g }
  +
  \sts{ Q_{\Omega} + g }_{xx}
  +
  \gamma\delta \sts{ Q_{\Omega}+g }
  +
  f\sts{ Q_{\Omega}+g }.
\end{align*}
Since $Q_{\Omega}$ is the solution of  \eqref{eqQ}, we have
\begin{align}
\label{eq3004}
  0
  =
  &
  \i g_{t}
  +
  \sts{\theta_{t}+\Omega}\sts{ Q_{\Omega} + g }
  +
  { g }_{xx} - \Omega { g }
  +
  \gamma\delta { g }
  +
  f\sts{ Q_{\Omega}+g }-f\sts{ Q_{\Omega} }.
\end{align}
By \eqref{op:L} and \eqref{eq:hot}, we have
\begin{align}
\label{eq3005}
  0
  =
  &
  \i g_{t}
  +
  \sts{\theta_{t}+\Omega}\sts{ Q_{\Omega} + g }
  +
  { g }_{xx} - \Omega { g }
  +
  \gamma\delta { g }
  +   f^{\prime}\sts{ Q_{\Omega}} g
  + \frac{1}{2}f^{\dprime}\sts{ Q_{\Omega}}\sts{ g, g}
  + \Rcal\sts{ Q_{\Omega},g }.
\end{align}

Finally, by taking $g\sts{t,x}=\lambda\sts{t} \phi_{\Omega}\sts{x}+\rho\sts{\lambda\sts{t}}Q_{\Omega}\sts{x}$ in \eqref{eq3005}, we can obtain \eqref{eq:epst}, this ends
the proof.
\end{proof}

By the orthogonal structure of the remainder term $\eps(t,x)$, we can obtain the dynamical control of the parameters $\lambda(t)$ and $\theta(t)$ as following.
\begin{lemm}
\label{lem:paradyn}
  Suppose $T>0$.
  There exist  $ 0<\tilde{\eta}_5\ll 1$ and $ 0<\tilde{\lambda}_5\ll 1$, such that
  if for all $t\in\left[ 0,T \right)$,  $\eps(t)$, $\theta(t)$, $\lambda(t)$ satisfying \eqref{eq:epst}, and 
  \begin{equation}
  \label{epstorth}
    \braket{ \eps\sts{t},\i Q_{\Omega} }=0 \quad\text{and}\quad \braket{\eps\sts{t},\varphi_{\Omega}}=0,
  \end{equation}
  and
  \begin{equation}
  \label{eq3sma}
    \norm{\eps\sts{t}}_{H^1}\leq \tilde{\eta}_5, \quad\text{and}\quad \abs{\lambda\sts{t}}\leq \tilde{\lambda}_5,
  \end{equation}
  then   for all $t\in\left[0,T\right)$, we have
  \begin{equation}
  \label{paradyncontr}
    \abs{\lambda_t}+\abs{\theta_t+\Omega}\leq C\sts{ \abs{\lambda}+\norm{\eps\sts{t}}_{H^1} },
  \end{equation}
 where $C$ is a constant which only depends on $Q_{\Omega}$.
\end{lemm}
\begin{proof}
  Multiplying \eqref{eq:epst} with $ Q_{\Omega} $ and $ \i\varphi_{\Omega} $ respectively, we have
\begin{align}
  \notag
  &
  \lambda_{t}\braket{\i\sts{ \varphi_{\Omega}+\frac{\d\rho}{\d\lambda}\sts{\lambda} Q_{\Omega} },Q_{\Omega}}
  +\sts{ \theta_{t} + \Omega }\braket{ \sts{ Q_{\Omega}+\lambda\varphi_{\Omega}+\rho\sts{\lambda}Q+\eps },Q_{\Omega} }
  \\
  \notag
  =
  &
  \braket{ \Lcal\sts{ \lambda\varphi_{\Omega}+\rho\sts{\lambda}Q_{\Omega}+\eps }, Q_{\Omega} }
  \\
  \label{eq3006}
  &
  -\braket{
    f\sts{Q_{\Omega}+\lambda\varphi_{\Omega}+\rho\sts{\lambda}Q_{\Omega}+\eps}
    -
    f\sts{Q_{\Omega}}
    -
    f^{\prime}\sts{Q_{\Omega}}\sts{\lambda\varphi_{\Omega}+\rho\sts{\lambda}Q_{\Omega}+\eps}, Q_{\Omega} }
\end{align}
and
\begin{align}
  \notag
  &
  \lambda_{t}\braket{\i\sts{ \varphi_{\Omega}+\frac{\d\rho}{\d\lambda}\sts{\lambda} Q_{\Omega} },\i\varphi_{\Omega}}
  +\sts{ \theta_{t} + \Omega }\braket{ \sts{ Q_{\Omega}+\lambda\varphi_{\Omega}+\rho\sts{\lambda}Q+\eps },\i\varphi_{\Omega} }
  \\
  \notag
  =
  &
  \braket{ \Lcal\sts{ \lambda\varphi_{\Omega}+\rho\sts{\lambda}Q_{\Omega}+\eps }, \i\varphi_{\Omega} }
  \\
  \label{eq3007}
  &
  -\braket{
    f\sts{Q_{\Omega}+\lambda\varphi_{\Omega}+\rho\sts{\lambda}Q_{\Omega}+\eps}
    -
    f\sts{Q_{\Omega}}
    -
    f^{\prime}\sts{Q_{\Omega}}\sts{\lambda\varphi_{\Omega}+\rho\sts{\lambda}Q_{\Omega}+\eps}, \i\varphi_{\Omega} }.
\end{align}

Let 
\begin{equation*}
\mathcal{F}(Q_{\Omega}, \lambda, \eps) =  f\sts{Q_{\Omega}+\lambda\varphi_{\Omega}+\rho\sts{\lambda}Q_{\Omega}+\eps}
-
f\sts{Q_{\Omega}}
-
f^{\prime}\sts{Q_{\Omega}}\sts{\lambda\varphi_{\Omega}+\rho\sts{\lambda}Q_{\Omega}+\eps}, 
\end{equation*}
then by \eqref{fexp2} and \eqref{func:rho}, $\mathcal{F}(Q_{\Omega}, \lambda, \eps) $ is a polynomial of at least one degree with respect to $\lambda$ and $\eps$.  By  \eqref{eq3sma}, we have 
\begin{equation}
\label{eq3008}
\abs{\braket{\mathcal{F}(Q_{\Omega}, \lambda, \eps) 
		, Q_{\Omega} }} \leq C(Q_{\Omega}) + \abs{\braket{\mathcal{F}(Q_{\Omega}, \lambda, \eps) 
		 , \i\varphi_{\Omega} }}  \leq C(Q_{\Omega}) \left(\abs{\lambda}+\norm{\eps\sts{t}}_{H^1}\right).
\end{equation}
In addition,  by \eqref{func:rho} and  \eqref{eq3sma},  we also have 
\begin{equation}
\label{eq3009}
\abs{\braket{ \Lcal\sts{ \lambda\varphi_{\Omega}+\rho\sts{\lambda}Q_{\Omega}+\eps }, Q_{\Omega} }}
+ \abs{ \braket{ \Lcal\sts{ \lambda\varphi_{\Omega}+\rho\sts{\lambda}Q_{\Omega}+\eps }, \i\varphi_{\Omega} }}  \leq C(Q_{\Omega}) \left(\abs{\lambda}+\norm{\eps\sts{t}}_{H^1}\right).
\end{equation}

Combining \eqref{eq3006}, \eqref{eq3007}, \eqref{eq3008} and \eqref{eq3009}, we can obtain the result.
\end{proof}

\section{Proof of  Theorem \ref{mainthm}}\label{sect:main}
\begin{proof}
We argue by contradiction and divide the proof of main theorem  into several steps.

\begin{enumerate}[label=\emph{{Step \arabic*.}},ref=\emph{{Step \arabic*}}]
\item \textit{Preparation of the initial data.} Firstly, we can choose  $0<\lambda_0< \tilde{\lambda}_0\ll 1$ sufficiently small such that  
$
\Mcal\sts{u_0}=\Mcal\sts{Q_{\Omega}},
$
where
  \begin{equation}
\label{initdata}
u_0\sts{x}=Q_{\Omega}\sts{x}+\lambda_0\varphi_{\Omega}\sts{x}+\widetilde{\rho}\sts{\lambda_0}Q_{\Omega}\sts{x},
\end{equation}
and 
$ \widetilde{\rho}(\lambda) $ is defined by \eqref{func:trho}. It is easy to check that 
 \begin{equation}
\label{lam0sml}
\norm{u_0 - Q_{\Omega}}_{H^1} = \norm{\lambda_0\varphi_{\Omega}+\widetilde{\rho}\sts{\lambda_0}Q_{\Omega}}_{H^1}< C\lambda_0.
\end{equation}
Assume that the solitary wave $Q_{\Omega}\e^{\i\Omega t}$ is orbitally stable in the energy space. By \Cref{def:stab}, for $\eta_0>0$ to be determined later,  there exists sufficiently small $\lambda_0$ such that  the solution $u(t)$  of \eqref{dnls} with initial data $u_0 \in \Ucal\sts{Q_{\Omega}, C\lambda_0}$ is  global, and $ u\sts{t}\in\Ucal\sts{Q_{\Omega}, \eta_0}$  for all $t>0$.

%
%
 \item \textit{Geometric decomposition of the solution $u\sts{t}$.} Let $\rho(\lambda)$ be defined by \eqref{func:rho}. By \Cref{lem:ift2} and the standard regularity argument
  in \cite{MM2001GAFA}, there exist two $\Ccal^1$ functions $\lambda$ and $\theta$ with respect to $t$ such that the remainder term 
   \begin{equation}
  \label{eps:dcp}
  \eps\sts{t,x}=u\sts{t,x}\e^{\i\theta\sts{t}}
  -
  \Big({
  	Q_{\Omega}\sts{x}
  	+
  	{\lambda\sts{t}}\varphi_{\Omega}\sts{x}
  	+
  	\rho({ {\lambda}\sts{t}}){Q_{\Omega}}\sts{x}
  }\Big)
  \end{equation}
  satisfies the equation
   \begin{align}
  \notag
  \i\eps_{t}
  =&
  -\i\lambda_{t}\sts{ \varphi_{\Omega}+\frac{\d\rho}{\d\lambda}\sts{\lambda} Q_{\Omega} }
  -\sts{ \theta_{t} + \Omega }\sts{ Q_{\Omega}+\lambda\varphi_{\Omega}+\rho\sts{\lambda}Q+\eps }
  \\
  \notag
  &
  + \Lcal\sts{ \lambda\varphi_{\Omega}+\rho\sts{\lambda}Q_{\Omega}+\eps }
  -\frac{1}{2}f^{\dprime}\sts{Q_{\Omega}}\sts{ \lambda\varphi_{\Omega}
  	+\rho\sts{\lambda}Q_{\Omega}+\eps , \lambda\varphi_{\Omega}+\rho\sts{\lambda}Q_{\Omega}+\eps }
  \\
  \label{eq:epst2}
  &
  -\Rcal\sts{ Q_{\Omega}, \lambda\varphi_{\Omega}+\rho\sts{\lambda}Q_{\Omega}+\eps},
  \end{align}
  where
  $f^{\dprime}\sts{Q_{\Omega}}$ is defined by \eqref{f2d},
  $\Lcal$ and $\Rcal$ are defined by \eqref{op:L} and \eqref{eq:hot} in Lemma \ref{lem:epst}, and  for all $t>0$, we have
\begin{align}
  \label{nvt}
      \braket{{\eps\sts{t}},{\i{Q}_{\Omega}}}=0,
      \quad
      \braket{{\eps\sts{t}},{\varphi_{\Omega}}}=0,
\end{align}
and
\begin{align}
  \label{se}
  \norm{\eps\sts{t}}_{H^1}+\abs{\lambda\sts{t}} + \abs{\theta(t)}  \leq C\eta_0.
\end{align}
By choosing $\lambda_0$ sufficiently small such that we have 
\begin{equation*}
0<\max\{1, C \} \; \eta_0<\min
\Set{ 
	\tilde{\eta}_0,
	\tilde{\eta}_1 , \tilde{\lambda}_1,
	\tilde{\eta}_3 , \tilde{\lambda}_3, 
	\tilde{\eta}_4, \tilde{\lambda}_4, 
	\tilde{\eta}_5, \tilde{\lambda}_5 
}.
  \end{equation*}
Moreover, by \Cref{lem:paradyn},  we have
\begin{equation}
  \label{paradyn}
    \abs{\lambda_t}+\abs{\theta_t+\Omega}\leq C\sts{ \abs{\lambda}+\norm{\eps\sts{t}}_{H^1} }.
  \end{equation}

Next, by the conservation law of mass and  \eqref{eps:dcp},  we have
\begin{equation}
\label{masscsv}
\Mcal\Big({ 
  Q_{\Omega} +   {\lambda\sts{t}}\varphi_{\Omega} 
  + \rho({ {\lambda}\sts{t}}){Q_{\Omega}}  +  \eps\sts{t}
}\Big)
=
\Mcal\sts{ Q_{\Omega} },
\end{equation} 
which together with \Cref{lem:epsQ} implies that
\begin{equation}
  \label{epstQ}
    \abs{ \braket{ \eps\sts{t},Q_{\Omega} } }
    \leq
    C\Big({ \norm{\eps\sts{t}}_{H^1}^2+\abs{\lambda\sts{t}}\norm{\eps}_{H^1} + \lambda\sts{t}^4 }\Big),
    \quad\text{for all }t>0.
  \end{equation}

\item \textit{Estimates of the the remainder term $\eps\sts{t}$ and the parameter $\lambda\sts{t}$.} 
Combining the aboved estimates,  we have the following  estimates of remainder term $\eps\sts{t}$ and the parameter $\lambda\sts{t}$ as a consequence of \Cref{lem:expand2}.
\begin{prop}
\label{prop:epsclam}
  Let $u_0$ be defined by \eqref{initdata} and $\eps\sts{t}$ be defined by \eqref{eps:dcp}. Then for all $t>0$,
  we have 
  \begin{equation}
  \label{lampres}
    \lambda\sts{t}\geq \frac{1}{2}\lambda_0,
  \end{equation}
  and
  \begin{equation}
  \label{eq:epsclam}
    \norm{\eps\sts{t}}_{H^1}^2\leq - \frac{2}{\kappa}\mathbf{d}^{\tprime}\sts{\Omega}\lambda\sts{t}^3,
  \end{equation}
  where $\kappa$ is the constant defined in \Cref{coro:coer}. 
\end{prop}
\begin{proof} The proof is similar to that of \cite[Proposition 4.1]{MTX2018}.  We give the details for the reader's convenience. 
	
	Firstly, as in the proof of Lemma \ref{lem:expand1}, we have 
	\begin{align*}
&\;	\Scal_{\Omega}\sts{
		u_0
	}
	-
	\Scal_{\Omega}\sts{
		Q_{\Omega}
	}\\
	=
	&\;
	\Scal_{\Omega}\sts{
		Q_{\Omega}+\lambda_0\phi_{\Omega}+\widetilde{\rho}\sts{\lambda_0}Q_{\Omega}
	}
	-
	\Scal_{\Omega}\sts{
		Q_{\Omega}
	}
	\\
	=
	&\;
	\frac{1}{2}
	\Scal^{\dprime}_{\Omega}\sts{Q_{\Omega}}
	\sts{
		\lambda_0\phi_{\Omega}+\widetilde{\rho}\sts{\lambda_0}Q_{\Omega},~
		\lambda_0\phi+\widetilde{\rho}\sts{\lambda_0}Q_{\Omega}
	}
	\\
	&
	+\frac{1}{6}
	\Scal^{\tprime}_{\Omega}\sts{Q_{\Omega}}
	\sts{
		\lambda_0\phi_{\Omega}+\widetilde{\rho}\sts{\lambda_0}Q_{\Omega},~
		\lambda_0\phi_{\Omega}+\widetilde{\rho}\sts{\lambda_0}Q_{\Omega},~
		\lambda_0\phi_{\Omega}+\widetilde{\rho}\sts{\lambda_0}Q_{\Omega}
	}
	\\
	&
	+
	\so{
		\norm{ \lambda_0\phi_{\Omega}+\widetilde{\rho}\sts{\lambda_0}Q_{\Omega}
		}_{H^1}^3
	}. \numberthis\label{ST1}
	\end{align*}
	where we used the fact that $\Scal'_{\Omega}\sts{Q_{\Omega}}=0$. By \eqref{func:trho} and  the fact that $\action{Q_{\Omega}}{\phi}=0$, we have
	\begin{align*}
	&\;
	\Scal^{\dprime}_{\Omega}\sts{Q_{\Omega}}
	\sts{
		\lambda_0\phi_{\Omega}+\widetilde{\rho}\sts{\lambda_0}Q_{\Omega},~
		\lambda_0\phi_{\Omega}+\widetilde{\rho}\sts{\lambda_0}Q_{\Omega}
	}
	\\
	=
	& \;
	\left(\lambda_0\right)^2
	\Scal^{\dprime}_{\Omega}\sts{Q_{\Omega}}
	\sts{\phi_{\Omega},~\phi_{\Omega}}
	+
	2\lambda_0\widetilde{\rho}\sts{\lambda_0}
	\Scal^{\dprime}_{\Omega}\sts{Q_{\Omega}}
	\sts{\phi_{\Omega},~ Q_{\Omega}}
	+
	\widetilde{\rho}\sts{\lambda_0}^2\Scal^{\dprime}_{\Omega}\sts{Q_{\Omega}}
	\sts{Q_{\Omega},~Q_{\Omega}}
	\\
	=
	&
	-
	\left(\lambda_0\right)^2
	\action{Q_{\Omega}}{\phi_{\Omega}}
	-
	2\lambda_0\widetilde{\rho}\sts{\lambda_0}
	\action{Q_{\Omega}}{Q_{\Omega}}
	+
	\widetilde{\rho}\sts{\lambda_0}^2\Scal^{\dprime}_{\Omega}\sts{Q_{\Omega}}
	\sts{Q_{\Omega},~Q_{\Omega}}
	\\
	=
	&
	-
	2\lambda_0\widetilde{\rho}\sts{\lambda_0}
	\action{Q_{\Omega}}{Q_{\Omega}}
	+
	\widetilde{\rho}\sts{\lambda_0}^2\Scal^{\dprime}_{\Omega}\sts{Q_{\Omega}}
	\sts{Q_{\Omega},~Q_{\Omega}}
	\\
	=
	&
	\left(\lambda_0\right)^3
	\action{\phi_{\Omega}}{\phi_{\Omega}}
	+
	\so{\abs{\lambda_0}^3},
	\end{align*}
which together with \eqref{ST1} and \eqref{eq:d3rdd} implies that
	\begin{align*}
	\Scal_{\Omega}\sts{
		u_0
	}
	-
	\Scal_{\Omega}\sts{
		Q_{\Omega}
	}
	=
	&  \left(
	\frac{1}{2}
	\action{\phi_{\Omega}}{\phi_{\Omega}}
	+\frac{1}{6}
	\Scal^{\tprime}_{\Omega}\sts{Q_{\Omega}}
	\sts{\phi_{\Omega},~\phi_{\Omega},~\phi_{\Omega}} \right) \cdot \left(\lambda_0\right)^3
	+
	\so{\abs{\lambda_0}^3}
	\\
	=
	& \;
	\frac{1}{6}~
	\mathbf{d}^{\tprime}(\Omega)\cdot \left( \lambda_0\right)^3
	+
	\so{\abs{\lambda_0}^3}.
	\numberthis\label{eST1}
	\end{align*}

	Secondly, by Lemma \ref{lem:expand2} and Lemma \ref{lem:restcoer}, we know that for any $t\geq 0$,  there exists some $\kappa>0$ such that
	\begin{align*}
	\Scal_{\Omega}\sts{
		u\sts{t}
	}
	-
	\Scal_{\Omega}\sts{
		Q_{\Omega}
	}
	= & \;
	\Scal_{\Omega}\sts{
		Q_{\Omega}
		+
		\lambda\sts{t}\phi_{\Omega}
		+
		\rho\sts{\lambda\sts{t}}Q_{\Omega}+\eps\sts{t}
	}
	-
	\Scal_{\Omega}\sts{
		Q_{\Omega}
	}
	\\
	= &\;
	\frac{1}{6}
	\mathbf{d}^{\tprime}(\Omega)\cdot \left(\lambda\sts{t}\right)^3
	+
	\Scal''_{\Omega}\sts{Q_{\Omega}}  \sts{\eps\sts{t},\eps\sts{t}}
	+
	\so{\abs{\lambda\sts{t}}^3}
	+
	\so{ \norm{\eps\sts{t}}_{H^1}^2 }
	\\
	\geq  &\;
	\frac{1}{6}
	\mathbf{d}^{\tprime}(\Omega)\cdot \left(\lambda\sts{t}\right)^3
	+
	\frac{\kappa}{4}\norm{\eps\sts{t}}_{H^1}^2
	+
	\so{\abs{\lambda\sts{t}}^3}
	+
	\so{ \norm{\eps\sts{t}}_{H^1}^2 }.
	\numberthis\label{eST2}
	\end{align*}
	
	Finally, by the mass and energy conservation laws, we have
	\begin{equation*}
	\Scal_{\Omega}\sts{u\sts{t}}=\Scal_{\Omega}\sts{u_0},
	\text{~~for any ~~} t\geq 0.
	\end{equation*}
	Therefore, by \eqref{eST1}, \eqref{eST2} and the fact that
	$\mathbf{d}^{\tprime}(\Omega)<0$, we have
	\begin{align*}
	\frac{1}{24}~
	\mathbf{d}^{\tprime}(\Omega)\cdot \left(\lambda_0\right)^3
	\geq
	& \;
	\frac{1}{6}~
	\mathbf{d}^{\tprime}(\Omega)\cdot \left(\lambda_0\right)^3
	+
	\so{\abs{\lambda_0}^3}
	\\
	\geq
	& \;
	\frac{1}{6}~
	\mathbf{d}^{\tprime}(\Omega)\cdot \left(\lambda\sts{t}\right)^3
	+
	\frac{\kappa}{4}~\norm{\eps\sts{t}}_{H^1}^2
	+
	\so{\abs{\lambda\sts{t}}^3}
	+
	\so{ \norm{\eps\sts{t}}_{H^1}^2 }
	\\
	\geq
	&
	\frac{1}{3}~
	\mathbf{d}^{\tprime}(\Omega)\cdot \left(\lambda\sts{t}\right)^3
	+
	\frac{\kappa}{6}~\norm{\eps\sts{t}}_{H^1}^2,
	\end{align*}
	which implies that
	\begin{equation*}
	\lambda\sts{t}\geq \frac{1}{2}\lambda_0,
	\text{~~ and ~~}
	\norm{\eps \sts{t}}_{H^1}^2
	\leq
	-
	\frac{2}{\kappa}~
	\mathbf{d}^{\tprime}(\Omega)\cdot \left(\lambda\sts{t}\right)^3.
	\end{equation*}
	This concludes the proof of Proposition \ref{prop:epsclam}.
	\end{proof}

By \eqref{epstQ} and  \eqref{eq:epsclam}, we have
\begin{equation}
\label{epstQclam}
  \abs{ \braket{ \eps\sts{t},Q_{\Omega} } }
  \leq 
  C \lambda(t)^{\frac{5}{2}},
\end{equation}
where $C$ is a constant independent of $\eps\sts{t}$ and $\lambda\sts{t}$.

\item\textit{Monotonicity formula.} Let us define 
\begin{equation}
\label{Phi}
  \Phi\sts{t,x}
  =
  \varphi_{\Omega}\sts{x}
  -\lambda\sts{t}
  \frac{\braket{\varphi_{\Omega},\varphi_{\Omega}}
  }{
    \braket{ Q_{\Omega},Q_{\Omega}}
  }
  Q_{\Omega}\sts{x},
\end{equation}
and the Virial type quantity as following
\begin{equation}
\label{virial}
  \Iscr\sts{t}
  =
  \braket{\i\eps\sts{t}, \Phi\sts{t}}.
\end{equation}
By \eqref{eq:epst2} and \eqref{nvt}, we have the following estimates 
\begin{align}
\label{It}  \frac{\d}{\d t}\Iscr\sts{t}
  = & 
  \braket{ {\i\partial_t \eps},{\Phi\sts{t}} }
  -
  \lambda_t
  \frac{\braket{\varphi_{\Omega},\varphi_{\Omega}}
  }{
    \braket{ Q_{\Omega},Q_{\Omega}}
  }
  \braket{{\i\eps },{Q_{\Omega}}}
\\
\notag
 = &  \braket{ {\i\partial_t \eps},{\Phi\sts{t}}}
  \\
  =
  &
  \label{lamt}
  -\lambda_{t}
    \braket{ \i\sts{ \varphi_{\Omega}+\frac{\d\rho}{\d\lambda}\sts{\lambda(t)} Q_{\Omega} }, {\Phi\sts{t}} }
  \\
  &
  \label{thett}
  -
  \sts{ \theta_{t} + \Omega }
  \braket{ \sts{ Q_{\Omega}+\lambda(t)\varphi_{\Omega}+\rho\sts{\lambda(t)}Q+\eps(t) }, {\Phi\sts{t}} }
  \\
  &
  \label{lint}
  + \braket{ \Lcal\sts{ \lambda(t)\varphi_{\Omega}+\rho\sts{\lambda(t)}Q_{\Omega}+\eps(t) }, {\Phi\sts{t}} }
  \\
  &
  \label{quadt}
  -\frac{1}{2}\braket{ f^{\dprime}\sts{Q_{\Omega}}
      \sts{ \lambda(t)\varphi_{\Omega}
          +\rho\sts{\lambda(t)}Q_{\Omega}+\eps(t) }\sts{ \lambda(t)\varphi_{\Omega}+\rho\sts{\lambda(t)}Q_{\Omega}+\eps (t)
      }
      ,
      {\Phi\sts{t}}
      }
  \\
  &
  \label{hott}
  -\braket{ \Rcal\sts{ Q_{\Omega}, \lambda(t)\varphi_{\Omega}+\rho\sts{\lambda(t)}Q_{\Omega}+\eps(t)}  , {\Phi\sts{t}}}.
\end{align}

\textit{Estimate of \eqref{lamt}.} By \eqref{Q} and \eqref{phi},  we have the vanishing result
\begin{align}
\label{lamtend}
  \eqref{lamt}
  =
  &
  0.
\end{align}

\textit{Estimate of \eqref{thett}.}
By \eqref{deg:Q} and \eqref{nvt}, we have
\begin{align}
\notag
  &
  \braket{ \sts{ Q_{\Omega}+\lambda(t)\varphi_{\Omega}+\rho\sts{\lambda(t)}Q+\eps(t) }, {\Phi\sts{t}} }
  \\
\notag
  =
  &
  \braket{ \sts{ Q_{\Omega}+\lambda(t)\varphi_{\Omega}+\rho\sts{\lambda(t)}Q+\eps(t) }, \varphi_{\Omega} }
  \\
\notag
  &
  -
  \lambda(t)
  \frac{\braket{\varphi_{\Omega},\varphi_{\Omega}} }{ \braket{ Q_{\Omega},Q_{\Omega}} }
  \braket{ \sts{ Q_{\Omega}+\lambda(t)\varphi_{\Omega}+\rho\sts{\lambda(t)}Q+\eps(t) }, Q_{\Omega} }
  \\
\notag 
  =
  &
  -\lambda(t)\rho\sts{\lambda(t)}\braket{\varphi_{\Omega},\varphi_{\Omega}}
  -\lambda(t)\frac{\braket{\varphi_{\Omega},\varphi_{\Omega}} }{ \braket{ Q_{\Omega},Q_{\Omega}} }
  \braket{ \eps(t), Q_{\Omega}}.
\end{align}
By  \eqref{func:rho} and \eqref{epstQclam},  we obtain
\begin{equation}
\label{thett2}
   \braket{ \sts{ Q_{\Omega}+\lambda\sts{t}\varphi_{\Omega}+\rho\sts{\lambda\sts{t}}Q+\eps(t) }, {\Phi\sts{t}} } 
  =\bo{ \lambda\sts{t}^{\frac{7}{2}} }.
\end{equation} 

Now, inserting \eqref{paradyn} and \eqref{thett2} into \eqref{thett}, we can obtain
\begin{equation}
\label{thettend}
  \eqref{thett}=\so{\lambda\sts{t}^2}.
\end{equation}

\textit{ Estimate of \eqref{lint}. }
Since $\Lcal$ is a self-adjoint operator, we deduced by \Cref{phi2Q}, \eqref{func:rho} and \eqref{deg:Q} that
\begin{align}
\notag
  &
  \braket{ \Lcal\sts{ \lambda\varphi_{\Omega}+\rho\sts{\lambda}Q_{\Omega}+\eps }, {\Phi\sts{t}} }
  \\
\notag
  =
  &
  \braket{ \Lcal\sts{ \lambda\varphi_{\Omega}+\rho\sts{\lambda}Q_{\Omega}+\eps }, \varphi_{\Omega} }
  -
  \lambda
  \frac{\braket{\varphi_{\Omega},\varphi_{\Omega}}}{\braket{ Q_{\Omega},Q_{\Omega}}}
  \braket{ \Lcal\sts{ \lambda\varphi_{\Omega}+\rho\sts{\lambda}Q_{\Omega}+\eps }, Q_{\Omega} }
    \\
\label{lint1}
    =
    &
    \frac{3}{2}\lambda^2  {\braket{\varphi_{\Omega},\varphi_{\Omega}}}
    -\braket{\eps,Q_{\Omega}}
    -\lambda \rho\sts{\lambda} \frac{\braket{\varphi_{\Omega},\varphi_{\Omega}}}{\braket{ Q_{\Omega},Q_{\Omega}}}
      \braket{ \Lcal Q_{\Omega},Q_{\Omega} }
      -\lambda \rho\sts{\lambda} \frac{\braket{\varphi_{\Omega},\varphi_{\Omega}}}{\braket{ Q_{\Omega},Q_{\Omega}}}
      \braket{ \Lcal Q_{\Omega},\eps }.
\end{align}
It follows from \eqref{func:rho}, \eqref{eq:epsclam} and \eqref{epstQclam} that
\begin{equation}
\label{lintend}
  \eqref{lint}=\frac{3}{2}\lambda(t)^2  {\braket{\varphi_{\Omega},\varphi_{\Omega}}}+\so{ \lambda(t)^2 }
\end{equation}

\textit{ Estimate of \eqref{quadt}. }
Note that
\begin{align}
\notag
  &
  \braket{ f^{\dprime}\sts{Q_{\Omega}}
      \sts{ \lambda(t)\varphi_{\Omega}
          +\rho\sts{\lambda(t)}Q_{\Omega}+\eps(t) } \sts{ \lambda(t)\varphi_{\Omega}+\rho\sts{\lambda(t)}Q_{\Omega}+\eps (t)
      }
      ,
      {\Phi\sts{t}}
      }
  \\
  \notag
  =
  &
  \braket{ f^{\dprime}\sts{Q_{\Omega}}
  \sts{ \lambda(t)\varphi_{\Omega}
      +\rho\sts{\lambda(t)}Q_{\Omega}+\eps(t) } \sts{ \lambda(t)\varphi_{\Omega}+\rho\sts{\lambda(t)}Q_{\Omega}+\eps (t)
  }
  ,
  { \varphi_{\Omega} }
  }
  \\
  \notag
  &
  -
  \lambda(t)
  \frac{\braket{\varphi_{\Omega},\varphi_{\Omega}}}{\braket{ Q_{\Omega},Q_{\Omega}}}
  \braket{ f^{\dprime}\sts{Q_{\Omega}}
  \sts{ \lambda(t)\varphi_{\Omega}
      +\rho\sts{\lambda(t)}Q_{\Omega}+\eps(t) } \sts{ \lambda(t)\varphi_{\Omega}+\rho\sts{\lambda(t)}Q_{\Omega}+\eps (t)
  }
  ,
  { Q_{\Omega} }
  }
  \\
  \notag
  =
  &\;
  \lambda(t)^2\braket{ f^{\dprime}\sts{Q_{\Omega}}  \varphi_{\Omega}  \varphi_{\Omega}   ,{ \varphi_{\Omega} } }
  \\
  &
  \label{quadt12}
  +
  2
  \braket{ f^{\dprime}\sts{Q_{\Omega}}\sts{ \lambda(t)\varphi_{\Omega} }\sts{ \rho\sts{\lambda(t)}Q_{\Omega}+\eps(t) } ,{ \varphi_{\Omega} } }
  \\
  \label{quadt13}
  &
  +
  \braket{ f^{\dprime}\sts{Q_{\Omega}}
    \sts{ \rho\sts{\lambda(t)}Q_{\Omega}+\eps (t)} \sts{ \rho\sts{\lambda(t)}Q_{\Omega}+\eps(t) } ,{ \varphi_{\Omega} } }
  \\
  \label{quadt14}
  &
  -
  \lambda(t)
  \frac{\braket{\varphi_{\Omega},\varphi_{\Omega}}}{\braket{ Q_{\Omega},Q_{\Omega}}}
  \braket{ f^{\dprime}\sts{Q_{\Omega}}
  \sts{ \lambda(t)\varphi_{\Omega}
      +\rho\sts{\lambda(t)}Q_{\Omega}+\eps(t) } \sts{ \lambda(t)\varphi_{\Omega}+\rho\sts{\lambda(t)}Q_{\Omega}+\eps(t) 
  }
  ,
  { Q_{\Omega} }
  }
\end{align}
By \eqref{func:rho} and \eqref{eq:epsclam}, we have
\begin{align*}
  \eqref{quadt12} &= \bo{ \lambda(t)\rho\sts{\lambda(t)}+\lambda(t)\norm{\eps(t)}_{H^1} } =\so{\lambda(t)^2},
  \\
  \eqref{quadt13} &= \bo{ \lambda(t)^4+\lambda(t)^2\norm{\eps(t)}_{H^1} + \norm{\eps(t)}_{H^1}^2 }=\so{\lambda(t)^2},
  \\
  \eqref{quadt14} &= \bo{ \lambda(t)^2+\lambda(t)^4+\norm{\eps(t)}_{H^1}^2 }=\so{\lambda(t)^2},
\end{align*}
which implies that 
\begin{equation}
\label{quadend}
  \eqref{quadt}
  =-
  \frac{1}{2}\lambda(t)^2\braket{ f^{\dprime}\sts{Q_{\Omega}}{ \varphi_{\Omega}, \varphi_{\Omega} } ,{ \varphi_{\Omega} } }
  +\so{\lambda(t)^2}.
\end{equation}

\textit{ Estimate of \eqref{hott}. }
By \eqref{func:rho}, \eqref{eq:hot} and \eqref{eq:epsclam}, we have
\begin{equation}
\label{hottend}
  \eqref{hott}=\bo{ \lambda(t)^3+\norm{\eps(t)}_{H^1}^2 }=\so{\lambda(t)^2}.
\end{equation}
Therefore, by summing up \eqref{lamtend}, \eqref{thettend}, \eqref{lintend}, \eqref{quadend} and \eqref{hottend}, we obtian from 
\begin{equation*}
  \eqref{It}
  =
  \frac{\lambda(t)^2}{2}
  \sts{ -
    \braket{ f^{\dprime}\sts{Q_{\Omega}}{ \varphi_{\Omega}  \varphi_{\Omega} } ,{ \varphi_{\Omega} } } 
    +
    3{\braket{\varphi_{\Omega},\varphi_{\Omega}}}
  }
  +
  \so{\lambda(t)^2}.
\end{equation*}
It follows from \eqref{F3rdd} and \Cref{lem:d3rdd} that
\begin{equation}
  \label{eIt}
   \eqref{It}=\frac{1}{2}\mathbf{d}^{\tprime}\sts{\Omega}\lambda(t)^2+\so{\lambda(t)^2}.
\end{equation}

\item \textit{Conclusion.}
On the one hand, by \eqref{virial} and \eqref{Phi}, we obtain that
$\norm{\eps\sts{t}}_{H^1}$ and
$\norm{\Phi\sts{t}}_{H^1}$
are uniformly bounded with respect to $t$. Therefore, by the
Cauchy-Schwarz inequality, we have
\begin{equation}\label{tIbd}
  \abs{\Iscr\sts{t}}
  \text{~~uniformly bounded with respect to~~}
  t.
\end{equation}

On the other hand, since $\mathbf{d}^{\tprime}\sts{\Omega}<0$,
by \eqref{lampres} and \eqref{eIt}, we have
\begin{align*}
\frac{\d}{\d t}\Iscr\sts{t}
=
\frac{1}{2}~\mathbf{d}^{\tprime}\sts{\Omega}\lambda^2\sts{t}
+
\so{ \lambda\sts{t}^2 }
 \leq
 \frac{1}{4}~\mathbf{d}^{\tprime}\sts{\Omega}\lambda\sts{t}^2
\leq
\frac{1}{16}~\mathbf{d}^{\tprime}\sts{\Omega} \left(\lambda_0\right)^2,
\end{align*}
by integrating the above inequality
over $[0, ~t)$, we can obtain that
\begin{align*}
\Iscr\sts{t}
=
&
\Iscr\sts{0}
+
\int_{0}^{t}\Iscr^{\prime}\sts{s}\d s
\leq
\Iscr\sts{0}
+
\frac{1}{16}~\mathbf{d}^{\tprime}\sts{\Omega} \left(\lambda_0\right)^2 t,
\end{align*}
which means that
\begin{equation*}
  \lim_{t\to +\infty}\Iscr\sts{t}=-\infty,
\end{equation*}
which is in contradiction with \eqref{tIbd}. 

Above all, we complete the proof of \Cref{mainthm}.
\end{enumerate}
\end{proof}

\appendix
\section{Proof of \Cref{lem:d3rdd}} 
\label{app:d3rdd}
\begin{proof}[Proof of \Cref{lem:d3rdd}]
(1) and (3) and the fact that $ \mathbf{d}^{\dprime}\sts{\Omega}=0 $ in (2) were proved in \cite{FOO2008AIHP}. Now we show (2).  For the convenience of the readers, we will give an alternative proof of the estimate $ \mathbf{d}^{\dprime}\sts{\Omega}=0 $.
Since $\omega>\frac{\gamma^2}{4}$, we denote 
\begin{gather*}
  \omega\sts{\lambda} = \frac{\lambda^2\gamma^2}{4},\quad\text{for~~}\lambda>1,
  \\
  m\sts{\omega}=\Mcal\sts{Q_{\omega}},
\end{gather*}
where $\Mcal\sts{Q_{\omega}}$ is defin by \eqref{mass},
and define
\begin{equation}
\label{m2g}
  g\sts{\lambda} = m\sts{\omega\sts{\lambda}}.
\end{equation}
It follows from \eqref{d1omega} that
\begin{equation}
\label{ap01}
  d^{\dprime}\sts{\omega}=\frac{\d m}{\d\omega }\sts{\omega}.
\end{equation}
A direct computation implies that
\begin{align}
\label{ap02}
  \frac{\d }{\d \lambda}g\sts{\lambda}
  =
  \frac{\d m}{\d \omega}\sts{\omega}\frac{\d \omega}{\d \lambda}\sts{\lambda}
  =\frac{\lambda\gamma^2}{2}\frac{\d m}{\d \omega}\sts{\omega},
\end{align}
\begin{align}
\notag
\label{ap03}
  \frac{\d^2 }{\d \lambda^2}g\sts{\lambda}
  =
  &
  \frac{\d m}{\d \omega}\sts{\omega}\frac{\d^2 \omega}{\d \lambda^2}\sts{\lambda}
  +
  \frac{\d^2 m}{\d \omega^2}\sts{\omega}\sts{ \frac{\d \omega}{\d \lambda}\sts{\lambda} }^2
  \\
  =
  &
  \frac{\gamma^2}{2}\frac{\d m}{\d \omega}\sts{\omega}
  +
  \frac{\lambda^2\gamma^4}{4} \frac{\d^2 m}{\d \omega^2}\sts{\omega}.
\end{align}
Combining \eqref{ap01} with \eqref{ap02}, since $\lambda>1$, we obtain that
\begin{center}
  $\mathbf{d}^{\dprime}\sts{\omega}=0$  if and only if $ \frac{\d }{\d \lambda}g\sts{\lambda}=0 $.
\end{center}
By \eqref{Q} and \eqref{m2g}, we have 
\begin{align*}
  g\sts{\lambda}=\int_{0}^{+\infty}\sts{ Q_{\omega\sts{\lambda}}\sts{x} }^2\d x
  =
  C\sts{p,\gamma}h\sts{\lambda}q\sts{\lambda},
\end{align*}
where $C\sts{p,\gamma}=\sts{ \frac{p+1}{8} }^{ \frac{2}{p-1} }\frac{4}{p-1}\gamma^{ \frac{4}{p-1}-1 }>0$,
\begin{align*}
  h\sts{\lambda}= \lambda^{\frac{4}{p-1}-1}, 
  \quad\text{and}\quad 
  q\sts{\lambda} = \int_{ \arctanh\sts{\frac{1}{\lambda}} }^{+\infty}\sech^{\frac{4}{p-1}}\sts{y}\d y.
\end{align*}
A direct calculation implies that
\begin{align*}
  \notag
  \frac{\d }{\d\lambda}g\sts{\lambda}
  =
  &
  C\sts{p,\gamma}
  \sts{ 
    \frac{\d h}{\d\lambda}\sts{\lambda}q\sts{\lambda} 
    +  
    \frac{\d q}{\d\lambda}\sts{\lambda}h\sts{\lambda}
  }
  \\
  =
  &
  C\sts{p,\gamma}
  \sts{
    \frac{5-p}{p-1} \lambda^{\frac{4}{p-1}-2}q\sts{\lambda}
    +
    \lambda^{\frac{4}{p-1}-1}\frac{ \sts{\lambda^2-1}^{ \frac{2}{p-1}-1 } }{ \lambda^{\frac{4}{p-1}} }
  }.
\end{align*}
Therefore, $\mathbf{d}^{\dprime}\sts{\Omega}=0$ if and only if $\left(\frac{\d }{\d\lambda}g\right)\sts{ \frac{ 2\sqrt{\Omega} }{\gamma} }=0$, 
i.e. $\frac{ 2\sqrt{\Omega} }{\gamma}$ satisfies
\begin{align}
\notag
  q\sts{ \frac{ 2\sqrt{\Omega} }{\gamma} }
  = & 
  \frac{p-1}{p-5}
  \sts{ \frac{ 2\sqrt{\Omega} }{\gamma} }^{ \frac{p-5}{p-1} }
  \sts{\sts{ \frac{ 2\sqrt{\Omega} }{\gamma} }^2-1}^{ \frac{ 3-p }{p-1} }
  \\ 
    \label{ap07}
  = & 
  \frac{p-1}{p-5}
  \sts{ \frac{ 2\sqrt{\Omega} }{\gamma} }^{ -1 + 2\frac{p-3}{p-1} }
  \sts{\sts{ \frac{ 2\sqrt{\Omega} }{\gamma} }^2-1}^{- \frac{ p-3 }{p-1} }
\end{align}
which coincide with the fact that \eqref{Omegaeq}, i. e.  we have $\mathbf{d}^{\dprime}\sts{\Omega}=0$. 

Next, we show that $\mathbf{d}^{\tprime}\sts{\Omega}<0$. By \eqref{ap03}, we have
\begin{center}
  $\mathbf{d}^{\tprime}\sts{\Omega}<0$ if and only if $\left(\frac{\d^2 }{\d\lambda^2}g\right)\sts{ \frac{ 2\sqrt{\Omega} }{\gamma}}<0$.
\end{center}
Since
\begin{align*}
  \frac{\d^2 }{\d\lambda^2}g\sts{\lambda}
  =
  C\sts{p,\gamma}
  \sts{ \frac{\d^2 h}{\d\lambda^2}\sts{\lambda}q\sts{\lambda} 
  +  
  \frac{\d^2 q}{\d\lambda^2}\sts{\lambda}h\sts{\lambda} 
  +
  2
  \frac{\d h}{\d\lambda }\sts{\lambda}\frac{\d q}{\d\lambda }\sts{\lambda}
  },
\end{align*}
it suffices to show that 
\begin{align*}
  \left(\frac{\d^2 h}{\d\lambda^2} q 
  +  
  \frac{\d^2 q}{\d\lambda^2} h  
  +
  2
  \frac{\d h}{\d\lambda}  \frac{\d q}{\d\lambda } \right) \sts{ \frac{ 2\sqrt{\Omega} }{\gamma} }
  <0.
\end{align*}
By the direct computations, we have
\begin{align}
  \label{ap11}
  &
  \frac{\d^2 h}{\d\lambda^2}\sts{\lambda}q\sts{\lambda} 
  +  
  \frac{\d^2 q}{\d\lambda^2}\sts{\lambda}h\sts{\lambda} 
  +
  2
  \frac{\d h}{\d\lambda }\sts{\lambda}\frac{\d q}{\d\lambda }\sts{\lambda}
  \\
  \notag
  =
  &
  \sts{ \frac{4}{p-1}-1 }\sts{ \frac{4}{p-1}-2 }\lambda^{ \frac{4}{p-1}-3 } q\sts{\lambda} 
  \\ \notag
  & +
  \lambda^{ \frac{4}{p-1}-1 }\frac{3-p}{p-1}\sts{\lambda^2-1 }^{ \frac{2}{p-1}-2 }2\lambda^{ 1-\frac{4}{p-1} }
  \\
  \notag
  &
  -
  \lambda^{ \frac{4}{p-1}-1 } \frac{4}{p-1} \sts{\lambda^2-1 }^{ \frac{2}{p-1}-1 }\lambda^{ -1-\frac{4}{p-1} }
  \\ \notag
  & +
  2\frac{ \frac{4}{p-1}-1 }{ \lambda^{ \frac{4}{p-1}-2 }\sts{ \lambda^2-1 }^{ \frac{2}{p-1}-1 }\lambda^{ -\frac{4}{p-1}} }.
\end{align}
By inserting \eqref{ap07} into \eqref{ap11}, we obtain
\begin{align*}
  \left(\frac{\d^2 h}{\d\lambda^2} q 
+  
\frac{\d^2 q}{\d\lambda^2} h  
+
2
\frac{\d h}{\d\lambda}  \frac{\d q}{\d\lambda } \right) \sts{ \frac{ 2\sqrt{\Omega} }{\gamma} }
  =
  \frac{6-2p}{p-1}\sts{ \frac{4\Omega}{\gamma^2}-1 }^{ \frac{2}{p-1}-2 }<0,
\end{align*}
since $p>5$. Therefore, we obtain $\mathbf{d}^{\tprime}\sts{\Omega}<0$.

Finally, by \eqref{d2omega} and \Cref{phi2Q}, we have
\begin{align*}
  \mathbf{d}^{\tprime}\sts{\Omega}
  =&
  \Scal^{\tprime}_{\Omega}\sts{Q_{\Omega}}
  \sts{ \varphi_{\Omega} , \varphi_{\Omega}, \varphi_{\Omega} }
  +
  2\Scal^{\dprime}_{\Omega}\sts{Q_{\Omega}}
  \sts{ \varphi_{\Omega} , \partial_{\Omega}\varphi_{\Omega} }  
  +
  \braket{ \varphi_{\Omega}, \varphi_{\Omega} }
  +
  2\braket{ \varphi_{\Omega}, \varphi_{\Omega} }
  +
  2\braket{ Q_{\Omega}, \partial_{\Omega}\varphi_{\Omega} }
  \\
  =&
  \Scal^{\tprime}_{\Omega}\sts{Q_{\Omega}}
  \sts{ \varphi_{\Omega} , \varphi_{\Omega}, \varphi_{\Omega} }
  +
  3\braket{ \varphi_{\Omega}, \varphi_{\Omega} }.
\end{align*}
This ends the proof of \Cref{lem:d3rdd}.
\end{proof}

\bibliographystyle{plain}

\end{document}